\nonstopmode
\documentclass[10pt, reqno]{amsart}
\usepackage{graphicx}
\usepackage{latexsym}
\usepackage{fancyhdr}
\usepackage{amsmath, amssymb, amsthm}
\usepackage[all]{xy}
\usepackage{pdflscape}
\usepackage{longtable}
\usepackage{rotating}
\usepackage{verbatim}
\usepackage{hyperref}
\hypersetup{
    linkcolor = blue,
    citecolor = blue,
    urlcolor  = blue,
    colorlinks = true,
}

\usepackage{cleveref}
\usepackage{subfigure}
\usepackage{mathrsfs}
\usepackage{mdwlist}
\usepackage{dsfont}
\usepackage{mathtools}
\usepackage{float}
\usepackage{color}
\usepackage{stmaryrd}
\usepackage{orcidlink}

\usepackage{tkz-base}
\usepackage{pgfplots}
\pgfplotsset{compat=newest}

\usepackage{etoolbox}

\definecolor{teal}{rgb}{0.0, 0.5, 0.5}

\newcounter{mnotecount}[section]

\newcommand{\rmnote}[1]{}

\allowdisplaybreaks

\DeclareFontFamily{U}{mathb}{\hyphenchar\font45}
\DeclareFontShape{U}{mathb}{m}{n}{
      <5> <6> <7> <8> <9> <10> gen * mathb
      <10.95> mathb10 <12> <14.4> <17.28> <20.74> <24.88> mathb12
      }{}
\DeclareSymbolFont{mathb}{U}{mathb}{m}{n}
\DeclareFontSubstitution{U}{mathb}{m}{n}

\theoremstyle{plain}
\newtheorem*{theorem*}{Theorem}
\newtheorem{theorem}{Theorem}[section]
\newtheorem*{lemma*}{Lemma}
\newtheorem{lemma}[theorem]{Lemma}
\newtheorem*{assumption*}{Assumption}

\newtheorem*{proposition*}{Proposition}
\newtheorem{proposition}[theorem]{Proposition}
\newtheorem*{corollary*}{Corollary}
\newtheorem{corollary}[theorem]{Corollary}
\newtheorem*{claim*}{Claim}

\newtheorem*{conjecture*}{Conjecture}

\newtheorem*{result*}{Result}

\theoremstyle{definition}
\newtheorem*{definition*}{Definition}
\newtheorem{definition}[theorem]{Definition}
\newtheorem*{example*}{Example}
\newtheorem{example}[theorem]{Example}

\newtheorem*{algorithm*}{Algorithm}
\newtheorem*{remark*}{Remark}
\newtheorem*{remarks*}{Remarks}
\newtheorem{remark}[theorem]{Remark}

\newtheorem*{question*}{Question}

\newtheorem*{convention*}{Convention}

\Crefname{l}{Lemma}{Lemmas}    
\Crefname{p}{Proposition}{Propositions}
\Crefname{t}{Theorem}{Theorems}
\Crefname{c}{Corollary}{Corollaries}
\Crefname{r}{Remark}{Remarks}
\Crefname{d}{Definition}{Definitions}
\Crefname{e}{Example}{Examples}
\Crefname{conv}{Convention}{Conventions}

\numberwithin{equation}{section}

\def\al{\alpha}
\def\be{\beta}
\def\ga{\gamma}
\def\de{\delta}
\def\ep{\epsilon}

\def\ze{\zeta}

\def\ka{\kappa}
\def\la{\lambda}

\def\rh{\rho}

\def\si{\sigma}

\def\vh{\varphi}

\def\om{\omega}

\def\Ga{\Gamma}

\def\Ps{\Psi}

\def\C{\mathbb{C}}

\def\N{\mathbb{N}}

\def\R{\mathbb{R}}

\def\cB{\mathcal{B}}
\def\cC{\mathcal{C}}

\def\cE{\mathcal{E}}

\def\cG{\mathcal{G}}
\def\cH{\mathcal{H}}

\def\cS{\mathcal{S}}

\def\sU{\mathscr{U}}

\def\p{\partial}
\def\id{\on{id}}

\def\<{\langle}
\def\>{\rangle}
\renewcommand{\o}{\circ}

\def\ol{\overline}

\let\on=\operatorname

\newcommand{\sr}[1]%
{\ifmmode{}^\dagger\else${}^\dagger$\fi\ifvmode
\vbox to 0pt{\vss
 \hbox to 0pt{\hskip\hsize\hskip1em
 \vbox{\hsize3cm\raggedright\pretolerance10000
 \noindent #1\hfill}\hss}\vss}\else
 \vadjust{\vbox to0pt{\vss%
 \hbox to 0pt{\hskip\hsize\hskip1em%
 \vbox{\hsize3cm\raggedright\pretolerance10000%
 \noindent #1\hfill}\hss}\vss}}\fi%
}

\def\A{\;\forall}
\def\E{\;\exists}

\def\ev{\on{ev}}

\providecommand{\mapsfrom}{\kern.2em%
\setbox0=\hbox{$\leftarrow$\kern-.10em\rule[0.26mm]{0.1mm}{1.3mm}}\box0%
\kern.3em}

\def\AS{\mathcal{AC}^\infty}
\def\AA{\mathcal{AC}^\omega}
\def\PA{\mathcal{PC}^\omega}
\def\AU{\mathcal{AE}^{\{\omega\}}}

\title{On spaces of arc-smooth maps}

\author[Armin Rainer]{Armin Rainer \orcidlink{0000-0003-3825-3313}}

\address{Institute for Statistics and Mathematical Methods in Economics, E105-04,
TU Wien, Wiedner Hauptstraße 8, 1040 Vienna, Austria}
\email{armin.rainer@tuwien.ac.at}

\begin{document}

\begin{abstract}
    It is well-known that a function on an open set in $\R^d$ is smooth if and only if 
    it is arc-smooth, i.e., its composites with all smooth curves are smooth. 
    In recent work, we extended this and related results (for instance a real analytic version) 
    to suitable closed sets, notably, 
    sets with H\"older boundary and fat subanalytic sets satisfying a necessary topological 
    condition. 
    In this paper, we prove that the resulting set-theoretic identities of function spaces 
    are bornological isomorphisms with respect to their natural locally convex topologies.
    Extending the results to maps with values in convenient vector spaces, 
    we obtain corresponding exponential laws. 
    Additionally, we show analogous results for special ultradifferentiable 
    Braun--Meise--Taylor classes.
\end{abstract}

\thanks{This research was funded 
    by the Austrian Science Fund (FWF) DOI 10.55776/PAT1381823.
    For open access purposes, the author has applied a CC BY public copyright license 
    to any author-accepted manuscript version arising from this submission.}
\thanks{\orcidlinkc{0000-0003-3825-3313}}

\maketitle



\section{Introduction}

A result of Boman \cite{Boman67} states that a function $f$ defined on an open subset $U$ of $\R^d$ is smooth ($\cC^\infty$) if and only if it is \emph{arc-smooth} ($\AS$), i.e., 
$f \o c$ is $\cC^\infty$ for each $\cC^\infty$ curve $c$ in $U$. Arc-smooth functions are meaningful on arbitrary nonempty subsets $X$ of $\R^d$ but a few assumptions 
are necessary in order to expect a result similar to Boman's. 
Let us assume that $X$ is closed and \emph{fat}, i.e., 
$X$ is contained in the closure of its interior; 
thus $X = \ol X = \ol{X^\o}$.
This is a natural assumption for our purpose:
for example, on the algebraic set $X=\{(x,y) \in \R^2 : x^3 = y^2\}$ the function 
$X \ni (x,y) \mapsto y^{1/3}$ is arc-smooth, by a theorem of Joris \cite{Joris82}, but it is not the restriction to $X$ of a $\cC^\infty$ function on $\R^2$. 
Moreover, we assume that $X$ is \emph{simple}, i.e., each $x \in X$ has a basis of neighborhoods $\sU$ such that $X^\o \cap U$ is connected for all $U \in \sU$.
This condition is needed to guarantee uniqueness of potential candidates for derivatives at boundary points.
The third assumption is a certain \emph{tameness} of $X$: we will suppose that $X$ is subanalytic or a H\"older set (see the definition in \Cref{ssec:Hoelder})

For simple closed fat subanalytic or H\"older sets $X \subseteq \R^d$ we proved in \cite{Rainer18} (see also \cite{Rainer:2021tr}) 
that the arc-smooth functions on $X$ are precisely the restrictions of $\cC^\infty$ functions on $\R^d$:
\begin{equation} \label{e:AS}
    \AS(X) = \cC^\infty(X).
\end{equation}
(This can be extended to sets definable in a polynomially bounded o-minimal expansion of the real field that admit smooth rectilinearization, see \cite{Rainer:2021tr}.)
It is false on infinitely flat cusps, see \cite[Example 10.4]{Rainer18}. 

The real analytic analogue of Boman's theorem is wrong: the composite of the function $f(x,y) = \frac{xy^2}{x^2+y^2}$  
with every real analytic curve $c$ in $\R^2$ is real analytic, but $f$ is not $\cC^1$. 
There even exist discontinuous functions $f$ 
with the property that $f\o c$ is real analytic for all real analytic curves $c$, see
\cite{BierstoneMilmanParusinski91}. 
But the arc-smooth functions $f$ on an open set $U \subseteq \R^d$, that additionally 
have the property that $f\o c$ is real analytic for all real analytic curves $c$ in $U$,    
are precisely the real analytic functions on $U$, by a theorem of Bochnak and Siciak \cite{Siciak70, BochnakSiciak71}.

We proved in \cite{Rainer18} (see also \cite{Rainer:2024aa}) that this is also true on simple closed fat subanalytic or H\"older sets $X \subseteq \R^d$
in the sense that the arc-smooth functions on $X$ that also map real analytic curves to real analytic curves ($\AA$) have a real analytic extension to an open neighborhood of $X$:
\begin{equation} \label{e:AA}
    \AA(X) = \cC^\om(X).
\end{equation}

In contrast to the case of open domains,
for closed domains there is a loss of regularity, namely, 
the discrepancy between $m$ and $n$, where $m$ is
the number of derivatives of $f \o c$ necessary to determine the $n$ first derivatives of $f$.
This loss of derivatives 
was linked in an exact way to the sharpness of the (outward pointing) cusps in the boundary of $X$ in \cite{Rainer:2021tr}.
Even at smooth parts of the boundary $2n$ derivatives of $f \o c$ are needed for the first $n$ derivatives of $f$.
See also \cite{Rainer:2024aa} for the real analytic case.

This loss of regularity manifests itself also in the ultradifferentiable framework of Denjoy--Carleman classes (as explored in \cite{Rainer18}) 
while on open sets we have an ultradifferentiable version of Boman's theorem, see \cite{KMRc}.  

In this paper, we work with special ultradifferentiable Braun--Meise--Taylor classes whose defining weight functions $\om$ 
have a property (see \eqref{e:robust})
that allows for absorption of the loss of derivatives.
We show that, for all simple fat closed subanalytic sets $X \subseteq \R^d$, 
the functions on $X$ that map $\cE^{\{\om\}}$ curves in $X$ to $\cE^{\{\om\}}$ curves ($\AU$) are restrictions of $\cC^\infty$ functions on $\R^d$ that satisfy the defining $\cE^{\{\om\}}$ bounds 
on compact subsets of $X$:
\begin{equation} \label{e:AU}
    \AU(X) = \cE^{\{\om\}}(X).
\end{equation}
It should be mentioned that the class $\cE^{\{\om\}}$ is non-quasianalytic and stable under composition. 
In the quasianalytic case, there is no hope for a result like this, 
see \cite{Jaffe16} and \cite{Rainer:2019aa}.

We will show that the set-theoretic identities \eqref{e:AS}, \eqref{e:AA}, and \eqref{e:AU} are bornological 
isomorphisms with respect to their natural locally convex topologies. 
Furthermore, we prove the bornological isomorphisms (\emph{exponential laws})
\begin{align*}
    \AS(X_1,\AS(X_2,E)) &\cong \AS(X_1 \times X_2,E), 
    \\
    \AA(X_1,\AA(X_2,E)) &\cong \AA(X_1 \times X_2,E),
    \\
    \AU(X_1,\AU(X_2,E)) &\cong \AU(X_1 \times X_2,E), 
\end{align*}
where $X_i \subseteq \R^{d_i}$, $i=1,2$, are arbitrary simple fat closed subanalytic sets and $E$ is any convenient vector space (see the definition in \Cref{ssec:convenient}).
To make sense of the left-hand sides (even if $E=\R$) we have to extend the definitions of $\AS$, $\AA$, and $\AU$ to maps with values in convenient vector spaces.
It turns out that the bornological isomorphisms \eqref{e:AS}, \eqref{e:AA}, and \eqref{e:AU} 
lift to versions for such vector valued maps. 
As a consequence we obtain the exponential laws  
\begin{align*}
    \cC^\infty(X_1,\cC^\infty(X_2,E)) &\cong \cC^\infty(X_1 \times X_2,E), 
    \\
    \cC^\om(X_1,\cC^\om(X_2,E)) &\cong \cC^\om(X_1 \times X_2,E),
    \\
    \cE^{\{\om\}}(X_1,\cE^{\{\om\}}(X_2,E)) &\cong \cE^{\{\om\}}(X_1 \times X_2,E). 
\end{align*}
Note that the product $X_1 \times X_2$ is a simple fat closed subanalytic set in $\R^{d_1} \times \R^{d_2}$ (see \Cref{l:prod}).

For open sets $X_i$, even $c^\infty$-open in convenient vector spaces (see \Cref{ssec:convenient}), the exponential laws are well known:
for $\cC^\infty$ by \cite{Froelicher80,Froelicher81,Kriegl82,Kriegl83},
for $\cC^\om$ by \cite{KrieglMichor90}, 
and in the ultradifferentiable case by \cite{KMRc,KMRq,KMRu,Schindl14a}. 
For convex sets $X_i$ in convenient vector spaces with nonempty $c^\infty$-interior,
similar results in the $\cC^\infty$ and $\cC^\om$ case were obtained by \cite{Kriegl97}.

Let us briefly describe the structure of the paper.
In \Cref{sec:pre}, we recall facts on \emph{convenient analysis} needed later on, in particular, the uniform boundedness principle 
which will be used frequently. Moreover, we define H\"older sets and list their most important properties.
\Cref{sec:smooth} is devoted to the $\cC^\infty$ case.
The real analytic case is treated in \Cref{sec:analytic}. 
Here (in \Cref{ssec:2plot}) we also investigate maps that respect $2$-dimensional real analytic plots (without presupposing smoothness) 
and obtain a corresponding exponential law. In general, this class of maps strictly contains all $\AA$ maps, but on open sets and Lipschitz sets 
in $\R^d$ the two classes coincide. In \Cref{ssec:hol}, we comment briefly on the holomorphic case.
The ultradifferentiable case $\cE^{\{\om\}}$ is studied in \Cref{sec:ultra}. 
We use a result which was proved for Denjoy--Carleman classes by \cite{BelottoBierstoneChow17} (see also \cite{ChaumatChollet99}).
An adaptation of their result to our setting is proved in the appendix, see \Cref{t:A1}.

\subsection{Notation}
For a locally convex vector space $E$, we denote by $E^*$ (resp.\ $E'$) 
the dual space consisting of all continuous (resp.\ bounded) linear functionals on $E$. 

If $\cS$ is a regularity class (e.g., $\cC^\infty$, $\cC^\om$, $\cE^{\{\om\}}$) 
and $X$ is a nonempty subset of $\R^d$, then $\cS(\R,X)$ denotes the set of $\cS$ curves $c : \R \to \R^d$ 
that lie in $X$, i.e., $c(\R) \subseteq X$.

The euclidean open ball in $\R^d$ with radius $r$ and center $a$ is denoted by $B(a,r)$ 
and $\ol B(a,r)$ denotes its closure.

For a map $f : X \times Y \to Z$ defined on a product, we denote by $f^\vee : X \to Z^Y$ the map defined by 
$f^\vee(x)(y) := f(x,y)$. Conversely, given $g : X \mapsto Z^Y$, the map $g^\wedge : X \times Y \to Z$ is 
defined by $g^\wedge(x,y):= g(x)(y)$.

For a map $f : X \to Y$ we have the push-forward $f_* : X^Z \to Y^Z$, $f_*(g) = f \o g$, and the pull-back 
$f^* : Z^Y \to Z^X$, $f^*(g)=g \o f$. 

The evaluation map $\ev : Y^X \times X \to Y$ is defined by $\ev(f,x) = f(x)$. For $x \in X$, $\ev_x : Y^X \to Y$ 
is given by $\ev_x(f) = f(x)$.

\section{Preliminaries} \label{sec:pre}

\subsection{Convenient vector spaces and \texorpdfstring{$c^\infty$}{c infty}-topology} \label{ssec:convenient}

Let us recall some of the fundamentals of \emph{convenient analysis}. The main reference is the book \cite{KM97}, 
see also the three appendices in \cite{KMRc} for a brief overview.

A locally convex vector space $E$ is called a \emph{convenient vector space} if it is \emph{$c^\infty$-complete}, i.e.,
a curve $c$ in $E$ is smooth if and only if $\ell \o c$ is smooth for all $\ell \in E^*$ (or equivalently $\ell \in E'$).
An equivalent condition is that each \emph{Mackey Cauchy sequence} $(x_n)$ (i.e., $\mu_{mn}(x_m-x_n)$ is bounded for some 
real sequence $\mu_{mn} \to \infty$)
converges in $E$.

Let $E$ be a locally convex vector space. The final topology with respect to all smooth curves in $E$ is called 
\emph{$c^\infty$-topology}. Equivalently, it is the final topology with respect to all \emph{Mackey convergent} 
sequences $x_n \to x$, i.e., there is a real positive sequence $\mu_n \to \infty$ such that 
$\mu_n (x_n-x)$ is bounded; in this case, we say that $x_n$ is $\mu_n$-convergent to $x$. 

In general, the $c^\infty$-topology is finer than the given locally convex topology and it is not a vector space topology.
On Fr\'echet spaces the two topologies coincide.

For smooth, real analytic, and holomorphic convenient analysis in convenient vector spaces, see \cite{KM97}, 
and for ultradifferentiable convenient analysis, see \cite{KMRc,KMRq,KMRu} and \cite{Schindl14a}. 
Let us point out (since this will be used several times) 
that (multi)linear maps between convenient vector spaces are smooth, real analyic, and of ultradifferentiable 
class $\cE^{\{\om\}}$ if and only if they are bounded (see \cite[5.5 and 11.13]{KM97} and \cite[Proposition 8.3]{KMRu}).

\subsection{Uniform boundedness principle}

Let $E$ be a locally convex space and let $\cS$ be a point separating set of bounded linear maps with common domain $E$.
Following \cite[5.22]{KM97}, we say that \emph{$E$ satisfies the uniform $\cS$-boundedness principle} if  
any linear map $T : F \to E$ on a convenient vector space $F$ is bounded provided that $\ell \o T$ is bounded for all $\ell \in \cS$.

By \cite[5.24]{KM97}, any locally convex space $E$ that is webbed satisfies the uniform $\cS$-boundedness principle 
for any point separating family $\cS \subseteq E'$.

For later reference, we recall the following stability result.

\begin{lemma}[{\cite[5.25]{KM97}}] \label[l]{l:stab}
    Let $\mathcal{F}$ be a set of bounded linear
    maps $f:E\to E_f$ between locally convex spaces, let $\mathcal{S}_f$
    be a point separating set of bounded linear maps on $E_f$ for every
    $f\in \mathcal{F}$, and let $\mathcal{S}:=
    \{g\o f: f\in \mathcal{F}, g\in \mathcal{S}_f\}$. If $\mathcal{F}$ generates the
    bornology and $E_f$ satisfies
    the uniform $\mathcal{S}_f$-boundedness principle for all $f\in \mathcal{F}$, then
    $E$ satisfies the uniform $\mathcal{S}$-boundedness principle.
\end{lemma}

\subsection{H\"older sets} \label{ssec:Hoelder}

A closed fat set $X \subseteq \R^d$ is a \emph{H\"older set} (resp.\ a \emph{Lipschitz set}) if its interior $X^\o$ 
has the \emph{uniform $\al$-cusp property} for some $\al \in (0,1]$ (resp.\ for $\al=1$), i.e., 
for each $x \in \p X$ there exist $\ep,h,r>0$ and $A \in \on{O}(d)$ such that $y + A\,\Ga^\al_d(r,h) \subseteq X^\o$ for all $y \in X \cap B(x,\ep)$,
where 
\[
    \Ga^\al_d(r,h) := \{ (x',x_d) \in \R^{d-1} \times \R : |x'|< r,\, h \cdot(\tfrac{|x'|}{r})^\al < x_d <h\} 
\]
is a truncated open $\al$-cusp of radius $r$ and height $h$. 
This is equivalent to $X^\o$ having $\al$-H\"older boundary 
(i.e., in local orthogonal coordinates, $X^\o = \{x_d > a(x')\}$ and $\p X^\o = \{x_d=a(x')\}$, where $a$ is an $\al$-H\"older function). 

H\"older sets are simple, $(1/\al)$-regular (if the $\al$-cusp property holds), and their $c^\infty$-topology 
coincides with the trace topology from $\R^d$; see \cite{Rainer18} and \cite{Rainer:2021tr} for details and examples.

\subsection{Subanalytic sets}
A subset $X$ of a real analytic manifold $M$ is called \emph{subanalytic} if each point in $M$
has a neighborhood $U$ such that $X \cap U$ is a projection of a relatively compact \emph{semianalytic} (i.e., locally described by finitely many analytic equations and 
inequalities) set. See e.g.\ \cite{BM88} which serves as a reference for the main properties of subanalytic sets.

\section{Smooth maps} \label{sec:smooth}

\subsection{Arc-smooth functions}
Let $X \subseteq \R^d$ be nonempty.
We equip the vector space 
\[
    \AS(X) = \{f : X \to \R : f_* \cC^\infty(\R,X) \subseteq \cC^\infty(\R,\R)\}
\]
with the initial locally convex structure with respect to the family of maps
\begin{equation} \label{e:str}
    \AS(X) \stackrel{c^*}{\longrightarrow} \cC^\infty(\R,\R), \quad c \in \cC^\infty(\R,X),
\end{equation}
where $\cC^\infty(\R,\R)$ carries the topology of compact convergence in each derivative separately.
Then the space $\AS(X)$ is $c^\infty$-closed in the product $\prod_{c\in \cC^\infty(\R,X)} \cC^\infty(\R,\R)$ and thus a convenient vector space.

\begin{lemma} \label[l]{l:ubA}
    $\AS(X)$ satisfies the uniform $\cS$-boundedness principle for $\cS=\{\on{ev}_{x} : x \in X\}$.
\end{lemma}

\begin{proof}
    The assertion follows from \Cref{l:stab} and the fact that $\cC^\infty(\R,\R)$ is a Fr\'echet space, and hence webbed;
    note that $\{\on{ev}_t \o c^* : c\in \cC^\infty(\R,X), t \in \R \}= \{\on{ev}_x : x \in X\}$.
\end{proof}

\begin{remark} \label[r]{r:Froelicher}
    Note that $(X,\cC^\infty(\R,X),\AS(X))$ is the unique Fr\"olicher space generated by the inclusion $X \to \R^d$; see \cite[23.1]{KM97}. 

    If $\R^d$ carries its natural diffeology, then $\AS(X)$ is the space of smooth maps $X \to \R$ in the category of diffeological spaces,
    where $X \subseteq \R^d$ is endowed with the subspace diffeology; see \cite{Iglesias-Zemmour:2013aa}.
    Indeed, a map $f : X \to \R$ is smooth if and only if $f \o p$ is $\cC^\infty$ for all $\cC^\infty$ maps $p : U \to \R^d$ with $p(U) \subseteq X$,
    where $U$ is an open subset of some $\R^n$. By Boman's theorem, these are precisely the functions $f \in \AS(X)$.
\end{remark}

\subsection{The space \texorpdfstring{$\cC^\infty(X)$}{Cinfty(X)}}

Let $X \subseteq \R^d$ be closed and nonempty.
Let 
\[
    \cC^\infty(X) := \{f : X \to \R: f= F|_X \text{ for some } F \in \cC^\infty(\R^d)\} 
\]
be endowed with the quotient topology; 
then $\cC^\infty(X)$ is a Fr\'echet space.

Let $X \subseteq \R^d$ be a H\"older set or a simple closed fat subanalytic set. 
Then $X$ is \emph{$\cC^\infty$ determining} (see \cite{Plesniak:1990aa}) in the sense that, 
for each $f \in \cC^\infty(\R^d)$, $f|_X = 0$ implies $\p^\al f|_X=0$ for all $\al \in \N^d$.
Moreover, every $g \in \cC^\infty(X^\o)$ such that all $\p^\al g$, $\al \in \N^d$, extend continuously to $X$ 
is the restriction of a $\cC^\infty$ function on $\R^d$, see \cite[Lemma 1.10]{Rainer18}.
It follows that $\cC^\infty(X)$ is isomorphic to the space of Whitney jets of class $\cC^\infty$
and, moreover, the topology is determined by the seminorms
\[
    \|f\|_{K,\ell} := \sup_{x \in K \cap X^\o} \sup_{|\al|\le \ell} |\p^\al f(x)|, \quad \ell \in \N,\, K \subseteq X \text{ compact}. 
\]
Furthermore, there is a continuous linear extension operator $E : \cC^\infty(X) \to \cC^\infty(\R^d)$, so that 
$E(f)|_X = f$; see \cite{Bierstone78} and also \cite{Frerick:2007aa}.

\begin{lemma} \label[l]{l:ubC}
    $\cC^\infty(X)$ satisfies the uniform $\cS$-boundedness principle for $\cS=\{\on{ev}_{x} : x \in X\}$.
\end{lemma}

\begin{proof}
    $\cC^\infty(X)$ is a Fr\'echet space and thus webbed. Clearly, the point evaluations are bounded and point separating on $\cC^\infty(X)$. 
\end{proof}

\subsection{Arc-smooth functions have smooth extensions}

\begin{theorem} \label[t]{t:smooth1}
    Let $X \subseteq \R^d$ be a H\"older set or a simple closed fat subanalytic set. Then
    \begin{equation} \label{e:ASt} 
        \AS(X) = \cC^\infty(X)
    \end{equation}
    and the identity is a bornological isomorphism.
\end{theorem}

\begin{proof}
    The set-theoretic identity \eqref{e:ASt} was established in \cite{Rainer18}, see also
    \cite{Rainer:2021tr}.

    Boundedness of the inclusion $\cC^\infty(X) \subseteq \AS(X)$ follows from \Cref{l:ubA}.
    We give an alternative argument that shows that the inclusion is even continuous.
    For any fixed $c \in \cC^\infty(\R,X)$ we have to show that the linear map $c^* : \cC^\infty(X) \to \cC^\infty(\R,\R)$ is continuous.
    Let $I= [-r,r]$ and $R>0$ such that $c(I) \subseteq B(0,R)$.
    Let $f \in \cC^\infty(X)$ and let $F := E(f) \in \cC^\infty(\R^d)$, 
    where $E : \cC^\infty(X) \to \cC^\infty(\R^d)$ is a continuous linear extension operator.
    Then (see e.g.\ \cite[Lemma A.5]{Parusinski:2024ab}), for all $k$, 
    \[
        \|c^* (F)\|_{C^k(I,\R^d)} \le C(k) \, \|F\|_{C^k(\ol B(0,R))} (1+ \|c\|_{C^k(I,\R^d)})^k
    \]
    where $\|g\|_{C^k(\ol U,\R^d)} := \max_{0\le j \le k} \sup_{x \in U} \|d^jg(x)\|_{L_j(\R^n,\R^d)}$ and 
    $U \subseteq \R^n$ is either $(-r,r) \subseteq \R$ or $B(0,R) \subseteq \R^d$.
    Now $c^* (F)=c^* (f)$ and there exist $C>0$, $\ell \in \N$, and a compact subset $K \subseteq X$ such that 
    \[
        \|F\|_{C^k(\ol B(0,R))} \le C\, \|f\|_{K,\ell}.
    \]
    This implies the assertion.

    To see that the inclusion $\AS(X) \subseteq \cC^\infty(X)$ is bounded, we have to check, by \Cref{l:ubC},
    that $\AS(X) \ni f \mapsto f(x) \in \R$ is bounded for each $x \in X$.
    This follows from \eqref{e:str}, using the constant curves $c_x : t \mapsto x$.
\end{proof}

\subsection{The vector valued case}

Let $X \subseteq \R^d$ be nonempty and $E$ a convenient vector space.
Consider
\[
    \AS(X,E) := \{f : X \to E : f_*\cC^\infty(\R,X) \subseteq \cC^\infty(\R,E)\},
\]
and equip $\AS(X,E)$ with the initial locally convex structure with respect to the family of maps 
\[
    \AS(X,E) \stackrel{\ell_* \o c^*}{\longrightarrow} \cC^\infty(\R,\R), \quad c\in \cC^\infty(\R,X), ~ \ell \in E^*.
\]
Then the space $\AS(X,E)$ is $c^\infty$-closed in the product $\prod_{c \in \cC^\infty(\R,X),\ell \in E^*} \cC^\infty(\R,\R)$ 
and thus a convenient vector space.

\begin{lemma} \label[l]{l:ubAE}
    $\AS(X,E)$ satisfies the uniform $\cS$-boundedness principle for $\cS=\{\on{ev}_{x} : x \in X\}$.
\end{lemma}

\begin{proof}
    This follows from \Cref{l:ubA} and \Cref{l:stab}.
\end{proof}

Let $X \subseteq \R^d$ be closed and nonempty. 
We define 
\[
    \cC^\infty(X,E) := \{ f : X \to E : \ell \o f \in \cC^\infty(X) \text{ for all } \ell \in E^*\}
\]
and endow $\cC^\infty(X,E)$ with the initial locally convex structure with respect to the family of maps 
\[
    \cC^\infty(X,E) \stackrel{\ell_*}{\longrightarrow} \cC^\infty(X), \quad \ell \in E^*.
\]
Then $\cC^\infty(X,E)$ is a convenient vector spaces since it is $c^\infty$-closed in the product $\prod_{\ell \in E^*} \cC^\infty(X)$.

\begin{lemma} \label[l]{l:ubCE}
    $\cC^\infty(X,E)$ satisfies the uniform $\cS$-boundedness principle for $\cS=\{\on{ev}_{x} : x \in X\}$.
\end{lemma}

\begin{proof}
    This follows from \Cref{l:ubC} and \Cref{l:stab}.
\end{proof}

We are ready to deduce a vector valued version of \Cref{t:smooth1}.

\begin{theorem} \label[t]{t:smooth2}
    Let $X \subseteq \R^d$ be a H\"older set or a simple closed fat subanalytic set and $E$ a convenient vector space. Then
    \begin{equation} \label{e:ASE}
        \AS(X,E) = \cC^\infty(X,E)
    \end{equation}
    and the identity is a bornological isomorphism.
\end{theorem}

\begin{proof}
    The set-theoretic identity \eqref{e:ASE} 
    follows from \Cref{t:smooth1}:
    \begin{align*}
        f \in \AS(X,E) 
   &\Leftrightarrow \forall \ell \in E^* ~\forall  c \in \cC^\infty(\R,X): \ell \o f \o c \in \cC^\infty(\R,\R) 
   \\
   &\Leftrightarrow \forall \ell \in E^* : \ell \o f  \in \AS(X) 
   \\
   &\Leftrightarrow \forall \ell \in E^* : \ell \o f  \in \cC^\infty(X) 
   \\
   &\Leftrightarrow  f \in \cC^\infty(X,E). 
    \end{align*}

    That the identity \eqref{e:ASE} is a bornological isomorphism is a consequence of the fact that 
    $\AS(X,E)$ and $\cC^\infty(X,E)$ both satisfy the uniform boundedness principle with respect to point evaluations, see \Cref{l:ubAE} and \Cref{l:ubCE}. 
\end{proof}

\begin{remark} \label[r]{r:derext}
    In the setting of \Cref{t:smooth2},
    each $f \in \AS(X,E)$ is of class $\cC^\infty$ in the interior $X^\o$ and all derivatives $(f|_{X^\o})^{(n)} : X^\o \to L^n(\R^d,E)$ extend 
    continuously to $X$. This can be seen by repeating the proof in \cite{Rainer18} for the vector valued case (see also \cite{Kriegl97}).
\end{remark}

\subsection{Exponential laws}

\begin{lemma} \label[l]{l:prod}
    Let $X_i \subseteq \R^{d_i}$, $i=1,2$, 
    be simple closed fat subanalytic sets. Then also $X_1 \times X_2$ is simple closed fat subanalytic. 
\end{lemma}

\begin{proof}
    Clearly, the product $X_1 \times X_2$ is closed and subanalytic.
    It is fat because
    \[
        \ol {(X_1 \times X_2)^\o} = \ol {X_1^\o \times X_2^\o} = \ol{X_1^\o} \times  \ol{X_2^\o} = X_1 \times X_2.
    \]
    Let us check that $X_1 \times X_2$ is simple. 
    Fix $(x_1,x_2) \in X_1 \times X_2$. Since $X_1$ and $X_2$ are simple, for $i=1,2$, 
    there exist bases of neighborhoods $\sU_i$ of $x_i$ in $X_i$ such that $U_i \cap X_i^\o$ is connected for all $U_i \in \sU_i$.
    Then $\{U_1 \times U_2 : U_1 \in \sU_1, U_2 \in \sU_2\}$ is a basis of neighborhoods of $(x_1,x_2)$ in $X_1 \times X_2$ and for all $U_1 \in \sU_1$, $U_2 \in \sU_2$,
    \[
        (U_1 \times U_2) \cap (X_1 \times X_2)^\o = (U_1 \cap X_1^\o) \times  (U_2 \cap X_2^\o)
    \]
    is connected.
\end{proof}

\begin{theorem} \label[t]{t:exp}
    Let $X_i \subseteq \R^{d_i}$, $i=1,2$, 
    be simple closed fat subanalytic sets.
    Let $E$ be a convenient vector space.
    Then the following exponential laws hold as bornological isomorphisms:
    \begin{enumerate}
        \item  $\AS(X_1,\AS(X_2,E)) \cong \AS(X_1 \times X_2,E)$;
        \item $\cC^\infty(X_1,\cC^\infty(X_2,E)) \cong \cC^\infty(X_1 \times X_2,E)$.
    \end{enumerate}
\end{theorem}

\begin{proof}
    (1) This is precisely the exponential law in the category of Fr\"olicher spaces (see \cite[23.2 and 23.4]{KM97}), 
    where $X$ is the Fr\"olicher space described in \Cref{r:Froelicher} and $E$ is the Fr\"olicher space generated by the bounded linear functionals (see \cite[23.3]{KM97}).

    (2) By \Cref{t:smooth2} we have a bornological isomorphism $\AS(X_2,E) = \cC^\infty(X_2,E)$ and thus a diffeomorphism in the category of Fr\"olicher spaces.
    Thus, again by \Cref{t:smooth2}, we have bornological isomorphisms
    \begin{align*}
        \AS(X_1,\AS(X_2,E)) = \AS(X_1,\cC^\infty(X_2,E)) = \cC^\infty(X_1,\cC^\infty(X_2,E)),
    \end{align*}
    and, using \Cref{l:prod},
    \begin{align*}
        \AS(X_1 \times X_2,E) = \cC^\infty(X_1 \times X_2,E). 
    \end{align*}
    Then (2) follows from (1).
\end{proof}

\begin{remark}
    In \Cref{t:exp}(1), the sets $X_i$ can actually be arbitrary.
    \Cref{t:exp}(2) remains valid (with the same proof) if $X_1$, $X_2$, and also $X_1 \times X_2$ are H\"older sets, 
    for instance, if $X_1 = \R^n$. 
\end{remark}

\begin{corollary}
    Let $X \subseteq \R^d$ be a H\"older set or a simple closed fat subanalytic set.
    Then we have a bornological isomorphism
    \[
        \cC^\infty(\R^n,\cC^\infty(X)) \cong \cC^\infty(\R^n \times X).
    \]
\end{corollary}

\section{Real analytic maps} \label{sec:analytic}

\subsection{The space \texorpdfstring{$\AA(X)$}{AComega(X)}}

Let $X \subseteq \R^d$ be nonempty.
We consider the vector space
\[
    \AA(X) := \{f \in \AS(X) : f_* \cC^\om(\R,X) \subseteq \cC^\om(\R,\R)\}
\]
and endow it with the initial locally convex structure with respect to the family of maps
\begin{align*}
    \AA(X) &\stackrel{c^*}{\longrightarrow} \cC^\infty(\R,\R), \quad c \in \cC^\infty(\R,X), 
    \\
    \AA(X) &\stackrel{c^*}{\longrightarrow} \cC^\om(\R,\R), \quad c \in \cC^\om(\R,X), 
\end{align*}
where $\cC^\om(\R,\R)$ carries the final locally convex topology with respect to the embeddings 
(restrictions) of all spaces of holomorphic maps $\vh : U \to \C$ with $\vh|_\R : \R \to \R$, where $U$ is a neighborhood
of $\R$ in $\C$, with their topology of compact convergence.
Then $\AA(X)$ is a convenient vector space.

\begin{lemma} \label[l]{l:ubAAn}
    $\AA(X)$ satisfies the uniform $\cS$-boundedness principle for $\cS=\{\on{ev}_{x} : x \in X\}$.
\end{lemma}

\begin{proof}
    This follows from \Cref{l:stab} since $\cC^\infty(\R,\R)$ and $\cC^\om(\R,\R)$
    satisfy the uniform boundedness principle with respect to point evaluations; see \cite[Theorem 11.12]{KM97}.
\end{proof}

\subsection{The space \texorpdfstring{$\cC^\om(X)$}{Comega(X)}} 

Let $X \subseteq \R^d$ be closed and nonempty.
Let $\cC^\om(X)$ be the set of all functions $f : X \to \R$ 
such that there exists an open neighborhood $U$ of $X$ in $\C^d$ and a holomorphic function $F : U \to \C$ 
such that $f = F|_X$. 
We topologize $\cC^\om(X)$ as the inductive limit 
of the Fr\'echet spaces of holomorphic functions $F \in\cH(U)$ such that $F(U \cap \R^d) \subseteq \R$,
where $U$ ranges over the directed set (with respect to inclusion) of open neighborhoods of $X$ in $\C^d$. 

\begin{lemma} \label[l]{l:ubAn}
    $\cC^\om(X)$ satisfies the uniform $\cS$-boundedness principle for $\cS=\{\on{ev}_{x} : x \in X\}$.
\end{lemma}

\begin{proof}
    $\cC^\om(X)$ is webbed since it is an inductive limit of webbed spaces, see \cite[52.13]{KM97}, 
    and $\cS$ is point separating.
\end{proof}

\subsection{The spaces \texorpdfstring{$\AA(X)$}{AComega(X)} and \texorpdfstring{$\cC^\om(X)$}{Comega(X)} coincide}

\begin{theorem} \label[t]{t:analytic}
    Let $X \subseteq \R^d$ be a H\"older set or a simple closed fat subanalytic set. Then
    \begin{equation} \label{e:AAt} 
        \AA(X) = \cC^\om(X)
    \end{equation}
    and the identity is a bornological isomorphism.
\end{theorem}

\begin{proof}
    The set-theoretic identity \eqref{e:AAt} was proved in \cite[Corollary 1.17]{Rainer18}, 
    see also \cite[Corollary 1.2]{Rainer:2024aa}.
    Then \Cref{l:ubAAn} and \Cref{l:ubAn} imply that it is a bornological isomorphism.
\end{proof}

\subsection{The vector valued case}

Let $X \subseteq \R^d$ be nonempty
and $E$ a convenient vector space.
We consider the space 
\[
    \AA(X,E) := \{f : X \to E : \ell \o f \in \AA(X) \text{ for all } \ell \in E^*\}  
\]
with the initial locally convex structure with respect to the family of maps 
\[
    \AA(X,E) \stackrel{\ell_*}{\longrightarrow} \AA(X), \quad \ell \in E^*.
\]
Then $\AA(X,E)$ is a convenient vector space.

\begin{lemma} \label[l]{l:ubAAnE}
    $\AA(X,E)$ satisfies the uniform $\cS$-boundedness principle for $\cS=\{\on{ev}_{x} : x \in X\}$.
\end{lemma}

\begin{proof}
    This follows from \Cref{l:stab} and \Cref{l:ubAAn}.
\end{proof}

For $X \subseteq \R^d$ closed and nonempty, we consider the space 
\[
    \cC^\om(X,E) := \{f : X \to E : \ell \o f \in \cC^\om(X) \text{ for all } \ell \in E^*\}  
\]
with the initial locally convex structure with respect to the family of maps 
\[
    \cC^\om(X,E) \stackrel{\ell_*}{\longrightarrow} \cC^\om(X), \quad \ell \in E^*.
\]
Then $\cC^\om(X,E)$ is a convenient vector space.

\begin{lemma} \label[l]{l:ubAnE}
    $\cC^\om(X,E)$ satisfies the uniform $\cS$-boundedness principle for $\cS=\{\on{ev}_{x} : x \in X\}$.
\end{lemma}

\begin{proof}
    This follows from \Cref{l:stab} and \Cref{l:ubAn}.
\end{proof}

We easily obtain a vector valued version of \Cref{t:analytic}.

\begin{theorem} \label[t]{t:analytic2}
    Let $X \subseteq \R^d$ be a H\"older set or a simple closed fat subanalytic set and $E$ a convenient vector space. 
    Then
    \begin{equation} \label{e:AAE}
        \AA(X,E) = \cC^\om(X,E)
    \end{equation}
    and the identity is a bornological isomorphism.
\end{theorem}

\begin{proof}
    This is an easy consequence of \Cref{t:analytic}, \Cref{l:ubAAnE}, and \Cref{l:ubAnE}.
\end{proof}

In the case that on the dual $E^*$ there exists a Baire topology for which the point evaluations are continuous
(for instance, if $E$ is a Banach space), 
the elements of \eqref{e:AAE} have holomorphic extensions.

\begin{theorem}
    Let $X \subseteq \R^d$ be a H\"older set or a simple closed fat subanalytic set and $E$ a convenient vector space. 
    Assume that on $E^*$ exists a Baire topology for which the point evaluations $\on{ev}_x$, $x \in E$, are continuous.
    Then for each $f \in \AA(X,E) = \cC^\om(X,E)$ there is an open neighborhood $U$ of $X$ in $\C^d$ and 
    a holomorphic map $F : U \to E_\C$ such that $F|_X = f$, where $E_\C$ is the complexification of $E$.
\end{theorem}

\begin{proof}
    Let $f \in \AA(X,E)$. Then $f \in \AS(X,E)$ and hence the derivatives $f^{(n)} : X \to L^n(\R^d,E)$ exist; see \Cref{r:derext}. 
    Fix $x \in \p X$ and consider the sequence $(\frac{1}{n!}f^{(n)}(x))_{n\ge 0}$.
    For each $\ell  \in E^*$, the composite $\ell \o f$ extends to a holomorphic function $F_\ell$ 
    defined on an open neighborhood of $X$, since $\AA(X,E) = \cC^\om(X,E)$, by \Cref{t:analytic2}. 
    For each $v \in \R^d$ and $y \in X^\o$ close enough to $x$, we have 
    $F_\ell^{(n)}(y)(v^n) = \ell(f^{(n)}(y)(v^n))$ and letting $y \to x$ we find $F_\ell^{(n)}(x)(v^n) = \ell(f^{(n)}(x)(v^n))$, by continuity.
    We conclude that the power series
    \[
        \sum_{n \ge 0} \ell\big(\tfrac{1}{n!}f^{(n)}(x)(v^n)\big) \, t^n = \sum_{n \ge 0} \tfrac{1}{n!} F_\ell^{(n)}(x)(v^n) t^n
    \]
    has positive radius of convergence.
    It follows from \cite[Theorem 25.1]{KM97} that the power series $\sum_{n \ge 0} \tfrac{1}{n!}f^{(n)}(x)(v^n)$ converges for $v$ in some 
    neighborhood of $0 \in \C^d$ and hence represents a holomorphic map $F_x$ in a neighborhood $U_x$ of $x$.
    For each $\ell \in E^*$ and each $n$, 
    \[
        (\ell \o F_x)^{(n)}(x) = \ell (f^{(n)}(x)) = F_\ell^{(n)}(x)
    \]
    so that the holomorphic functions $\ell \o F_x$ and $F_\ell$ coincide on a neighborhood $U_x$ of $x$. 
    Thus 
    if $y \in X^\o \cap U_x$ then $\ell (F_x(y)) = F_\ell(y) = \ell (f(y))$.
    Since this holds for all $\ell \in E^*$, we conclude that the holomorphic extension $F_x : U_x \to E_\C$ satisfies $F_x|_{X^\o \cap U_x} = f|_{X^\o \cap U_x}$.

    Now it remains to check that $f$ and the $F_x$ glue coherently to a holomorphic extension of $f$
    which follows from the arguments in the proof of \cite[Proposition 2.2]{Rainer:2024aa}.
\end{proof}

\subsection{Exponential laws}

The following lemma is a special case of \cite[25.11]{KM97}.

\begin{lemma} \label[l]{l:mixing}
    Let $X \subseteq \R^d$ be nonempty.
    If $f : \R \times X \to \R$ is locally the restriction of a holomorphic map and $c \in \cC^\infty(\R,X)$,
    then $c^* \o f^\vee : \R \to \cC^\infty(\R,\R)$ is real analytic.
\end{lemma}

\begin{theorem} \label[t]{t:expAn}
    Let $X_i \subseteq \R^{d_i}$, $i=1,2$, 
    be simple closed fat subanalytic sets.
    Let $E$ be a convenient vector space.
    Then the following exponential laws hold as bornological isomorphisms:
    \begin{enumerate}
        \item  $\AA(X_1,\AA(X_2,E)) \cong \AA(X_1 \times X_2,E)$;
        \item $\cC^\om(X_1,\cC^\om(X_2,E)) \cong \cC^\om(X_1 \times X_2,E)$.
    \end{enumerate}
\end{theorem}

\begin{proof}
    (1) 
    Let $f \in \AA(X_1, \AA(X_2,E))$ and consider the associated map $f^\wedge : X_1 \times X_2 \to E$.
    Let $c = (c_1,c_2) : \R \to X_1 \times X_2$ be $\cC^a$, for $a=\infty$ or $a =\om$, and $\ell \in E^*$.
    It suffices to show that $\ell \o f^\wedge \o (c_1 \times c_2) : \R \times \R \to \R$ is $\cC^a$.
    Now $(\ell \o f^\wedge \o (c_1 \times c_2) )^\vee = \ell_* \o c_2^* \o f \o c_1 : \R \to \cC^a(\R,\R)$ 
    is of class $\cC^a$ by assumption.
    By the $\cC^a$ exponential law (on open domains), see \cite[3.12 and 11.18]{KM97}, we may conclude that $f^\wedge \in \AA(X_1 \times X_2,E)$.

    Conversely, let $f \in \AA(X_1 \times X_2,E)$.
    Then $f^\vee : X_1 \to \AA(X_2,E)$ is well-defined; 
    we want to show that this map is of class $\AA$.
    Let $c_1 \in \cC^a(\R, X_1)$, for $a=\infty$ or $a =\om$,
    $c_2 \in \cC^b(\R, X_2)$, for $b=\infty$ or $b =\om$,
    and $\ell \in E^*$.
    We have to prove that $\ell_* \o c_2^* \o f^\vee \o c_1 : \R \to \cC^b(\R,\R)$ is $\cC^a$.
    We distinguish three cases:

    (i) If $a=b$, the desired property follows from the $\cC^a$ exponential law (on open domains).

    (ii) Now assume that $c_1 \in \cC^\om$ and $c_2 \in \cC^\infty$.
    Then $\ell \o f \o (c_1 \times \id) : \R \times X_2 \to \R$ is of class $\cC^\om$.
    By \Cref{t:analytic}, $\ell \o f \o (c_1 \times \id)$ has a holomorphic extension.
    By \Cref{l:mixing},
    $c_2^* \o (\ell \o f \o (c_1 \times \id))^\vee = \ell \o c_2^* \o f^\vee \o c_1 : \R \to \cC^\infty(\R,\R)$
    is of class $\cC^\om$.

    (iii) Finally, let $c_1 \in \cC^\infty$ and $c_2 \in \cC^\om$.
    We may apply case (ii) to the map $\tilde f(x,y):= f(y,x)$ to conclude that
    $\ell \o c_1^* \o (\tilde f)^\vee \o c_2 : \R \to \cC^\infty(\R,\R)$
    is of class $\cC^\om$. By \cite[11.16]{KM97}, this means that 
    $\ell \o c_2^* \o f^\vee \o c_1 : \R \to \cC^\om(\R,\R)$
    is of class $\cC^\infty$.

    Thus we have proved that $f \in \AA(X_1,\AA(X_2,E))$ if and only if $f^\wedge \in \AA(X_1 \times X_2,E)$.
    That it is a bornological isomorphism is seen as follows. By \Cref{l:ubAAnE} 
    applied to $\AA(X_1 \times X_2,E)$,
    the linear map $(\cdot)^\wedge : \AA(X_1,\AA(X_2,E)) \to \AA(X_1 \times X_2,E)$ is bounded provided that
    \[
        f \mapsto \ell_*\o \ev_{(x_1,x_2)} \o (\cdot)^\wedge(f) = \ell(f^\wedge (x_1,x_2))
    \]
    is bounded for all $(x_1,x_2) \in X_1 \times X_2$ and $\ell \in E^*$. 
    This follows from the definition of the structure on $\AA(X_1,\AA(X_2,E))$, 
    viewing $x_1$ and $x_2$ as constant ($\cC^\infty$ or $\cC^\om$) curves and noting that $\ell_* \o x_2^*$ 
    is a continuous linear functional on $\AA(X_2,E)$.

    To see that the linear map $(\cdot)^\vee : \AA(X_1 \times X_2,E) \to \AA(X_1,\AA(X_2,E))$ is bounded
    we have to check, by \Cref{l:ubAAnE}, that 
    \[
        g \mapsto  \ev_{x_1} \o (\cdot)^\vee(g) = g^\vee(x_1) \in \AA(X_2, E)
    \]
    is bounded for all $x_1 \in X_1$,
    or equivalently, again by \Cref{l:ubAAnE}, that 
    \[
        g \mapsto \ell_*\o \on{ev}_{x_2} (g^\vee(x_1)) = \ell(g^\vee(x_1)(x_2))= \ell(g(x_1,x_2))
    \]
    is bounded for all $x_1 \in X_1$, $x_2 \in X_2$, and $\ell \in E^*$. 
    As before, it follows from the structure on $\AA(X_1\times X_2,E)$ and viewing $x_1 \times x_2$ as a constant 
    ($\cC^\infty$ or $\cC^\om$) curve in $X_1 \times X_2$.

    (2) This follows from (1), \Cref{l:prod}, and \Cref{t:analytic2}, by arguments similar to those 
    in the proof of \Cref{t:exp}(2). 
\end{proof}

\begin{remark} \label[r]{r:can}
    Note that in (1) we could replace one of the $X_i$ with any convenient vector space; 
    the proof of the fact that $f \mapsto f^\wedge$ is a bijection remains the same (using \cite[11.17 or 25.11]{KM97} instead of \Cref{l:mixing}). 
    Then the boundedness (even if both $X_i$ are subanalytic) can be deduced 
    as in \cite[Corollary 3.13]{KM97}.
\end{remark}

\begin{remark}
    \Cref{t:expAn} remains valid (with the same proof) if $X_1$, $X_2$, and also $X_1 \times X_2$ are H\"older sets, 
    for instance, if $X_1 = \R^n$. 
    Notice that \Cref{t:analytic} was used in the proof of (1) (and of (2)). Indeed, (1) is not always true as seen in the following example.
\end{remark}

\begin{example}[{\cite[25.12]{KM97}}] \label[e]{ex:explawfails}
    Let $X \subseteq \R^2$ be the graph of $h : \R \to \R$, $h(t):= \exp(-t^{-2})$ if $t\ne0$ and $h(0):= 0$
    and $f : \R \times X \to \R$ be defined by $f(x,y,z) := \frac{z}{x^2+y^2}$ for $x^2+y^2\ne 0$ and $f(0,0,z):=0$.
    Then $f \in \AA(\R \times X)$  but $f^\vee : \R \to \AA(X,\R)$ is not $\cC^\om$.
    Indeed, for $c(t): = (t,h(t))$ the map $c^* \o f^\vee : \R \to \cC^\infty(\R,\R)$, $x \mapsto (y \mapsto f(x,c(y)) = \frac{h(y)}{x^2+y^2})$
    is not $\cC^\om$. For details see \cite[25.12]{KM97}.
\end{example}

\subsection{Real analyticity on \texorpdfstring{$2$}{2}-dimensional plots} \label{ssec:2plot}

Instead of asking that a map respects $\cC^\infty$ and $\cC^\om$ curves we will now assume that it respects $2$-dimensional $\cC^\om$ plots.

It is known that
a function $f$ on an open nonempty set $U \subseteq \R^d$ 
is real analytic if and only if the restriction of $f$ to each affine $2$-plane that meets $U$ is real analytic, 
by \cite{Bochnak:1971aa,Bochnak:2018aa}. See also
\cite{Bochnak:2020tz} for a global version of this result.

For any nonempty subset $X \subseteq \R^d$ and any convenient vector space $E$ let us consider
\[
    \PA(X,E) := \{f : X \to E : \ell \o f \o p \in \cC^\om(\R^2,\R) \text{ for all } p \in \cC^\om(\R^2,X),\, \ell \in E^*\}
\]
endowed with the initial locally convex structure with respect to the family
\[
    \PA(X,E) \stackrel{\ell_* \o p^*}{\longrightarrow} \cC^\om(\R^2,\R), \quad  p \in \cC^\om(\R^2,X),\, \ell \in E^*,
\]
where $\cC^\om(\R^2,\R)$ carries the usual structure defined in analogy to the structure on $\cC^\om(\R,\R)$.
Then $\PA(X,E)$ is a convenient vector space.

\begin{lemma} \label[l]{l:ubPAE}
    $\PA(X)$ satisfies the uniform $\cS$-boundeness principle for $\cS=\{\on{ev}_{x} : x \in X\}$.
\end{lemma}

\begin{proof}
    This follows from \Cref{l:stab} since $\cC^\om(\R^2,\R)$ satisfies the uniform 
    $\cS$-boundedness principle; 
    see \cite[Theorem 11.12]{KM97}.
\end{proof}

\begin{theorem} \label[t]{t:expPA}
    Let $X_i \subseteq \R^{d_i}$, $i=1,2$, be arbitrary nonempty sets and $E$ a convenient vector space. 
    Then the exponential law holds as bornological isomorphism:
    \[
        \PA(X_1,\PA(X_2,E)) \cong \PA(X_1 \times X_2,E).
    \]
\end{theorem}

\begin{proof}
    The proof of the fact that $f \in \PA(X_1,\PA(X_2,E))$ if and only if $f^\wedge \in \PA(X_1 \times X_2,E)$ 
    can be reduced,
    by the definition of the convenient structure, to the case $X_1=X_2 = \R^2$ and $E = \R$,
    which then is a simple instance of the real analytic exponential law on open domains (see \cite[11.18]{KM97}). 

    That this gives a bornological isomorphism can be seen using the uniform boundedness principle, \Cref{l:ubPAE},
    similarly as in the proof of \Cref{t:expAn}(1).
\end{proof}

Let us now compare the spaces $\AA(X,E)$ and $\PA(X,E)$.

\begin{lemma} \label[l]{l:AAinPA}
    Let $X \subseteq \R^d$ be nonempty and $E$ a convenient vector space.
    We always have the bounded inclusion 
    \begin{equation} \label{e:PA1}
        \AA(X,E) \subseteq \PA(X,E).
    \end{equation}
    There exist $X$ and $E$ such that this inclusion is strict.
\end{lemma}

\begin{proof}
    Let $f \in \AA(X,E)$, $p \in \cC^\om(\R^2,X)$, and $\ell \in E^*$.
    We have to show that $g:=\ell \o f \o p \in \cC^\om(\R^2,\R)$ which holds since 
    $g$ respects $\cC^\infty$ and $\cC^\om$ curves in $\R^2$.

    Boundedness of the inclusion follows from \Cref{l:ubPAE}.

    Let $X:= \{(x,y) \in \R^2 : x \ge 0,~ x^{\sqrt 2} \le y \le x^{\sqrt 2} + x^2\}$.
    Then $f|_{X \setminus \{(0,0)\}} := 0$ and $f(0,0):=1$ belongs to $\PA(X)$ but not to $\AA(X)$. Indeed,
    any $p \in \cC^\om(\R^2,X)$ such that $(0,0) \in p(\R^2)$ must be constant, which follows from 
    \cite[Example 6.7]{Rainer:2021tr}. 
    On the other hand, for the $\cC^\infty$ curve $c : \R \to X$, 
    \[
        c(t) := \big(e^{-1/t^2}, e^{-\sqrt{2}/t^2}+ \tfrac{1}2 e^{-2/t^2}\big), \quad \text{ if } t\ne 0, 
        \quad c(0):= (0,0),
    \]
    the composite $f \o c$ is discontinuous.
    (One can argue similarly for the set defined in \Cref{ex:explawfails}.)
\end{proof}

We will see that \eqref{e:PA1} is an equality for open sets and Lipschitz sets $X \subseteq \R^d$. 
To this end we recall a result of Bochnak and Siciak.

\begin{theorem}[{\cite{Bochnak:2018aa}}] \label{thm:BSfindim}
    Let $U \subseteq \R^d$ be an open set, where $d\ge 2$.
    Let $f : U \to \R$ be a function.
    If the restriction $f|_{U\cap P}$ is real analytic for each affine $2$-plane $P$ in $E$, then $f$ is real analytic.
\end{theorem}

\begin{remark}
    The result is true for open subsets $U$ of infinite dimensional Banach spaces $E$ and \emph{continuous} functions $f : U \to \R$.
    By \cite[Theorem 7.5]{BochnakSiciak71}, it is enough to check that $f$ is infinitely Gateaux differentiable and
    $f|_{U\cap L}$ is real analytic for each affine line $L$ in $E$.
    That $f$ is infinitely Gateaux differentiable follows from \cite[Theorem 4]{Bochnak_1971}:
    by the finite dimensional case \Cref{thm:BSfindim}, $f|_{U \cap V}$ is real analytic for each finite-dimensional subspace $V$ of $E$.
\end{remark}

\begin{theorem} \label[t]{t:AAPA}
    Let $X \subseteq \R^d$ be open or a Lipschitz set and $E$ a convenient vector space.
    Then
    \begin{equation} \label{e:PA2}
        \AA(X,E)  = \PA(X,E)
    \end{equation}
    and the identity is a bornological isomorphism.
\end{theorem}

\begin{proof}
    The set-theoretic identity \eqref{e:PA2} follows, after composing with $\ell \in E^*$, from \Cref{l:AAinPA} and \Cref{thm:BSfindim} for open $X$ and 
    \cite[Theorem 1.3]{Rainer:2024aa} for Lipschitz sets.

    That the identity is a bornological isomorphism is a consequence of the fact that both 
    sides satisfy the uniform boundedness principle with respect to point evaluations; see \Cref{l:ubAAnE} and \Cref{l:ubPAE}.
\end{proof}

\begin{remark} \label[r]{r:AAPA}
    The theorem remains true if $X \subseteq \R^d$ is any simple closed set such that for each $z \in \p X$ there is a nondegenerate simplex $X_z$ 
    with $z \in X_z \subseteq X$. 
    In dimension $2$, it holds if this condition is fulfilled with compact fat subanalytic $X_z$, in particular, it is true for all
    H\"older sets and all simple closed fat subanalytic sets in $\R^2$. See \cite{Rainer:2024aa}.
\end{remark}

\begin{remark}
    \Cref{t:expPA} and \Cref{t:AAPA} together yield a shorter proof of the $\AA$ exponential law, see \Cref{t:expAn}(1), 
    in the cases where $X_1$, $X_2$, and $X_1 \times X_2$ satisfy the assumptions of \Cref{t:AAPA} or \Cref{r:AAPA}.
\end{remark}

\subsection{Remarks on holomorphic maps} \label{ssec:hol}

One cannot expect that holomorphic curves can detect holomorphy even on nice closed sets in $\C^d$ 
(by the open mapping theorem for holomorphic maps).

\begin{theorem}
    Let $X \subseteq \C^d \cong \R^{2d}$ be a H\"older set or a simple closed fat subanalytic set 
    and $E$ a complex convenient vector space. 
    If $f \in \AA(X,E)$ is such that its derivative $f'(x)$ is $\C$-linear for all $x \in X^\o$, then $f$ has a holomorphic extension to some 
    open neighborhood of $X$ in $\C^d$.
\end{theorem}

\begin{proof}
    By \Cref{t:analytic2}, $f$ has a real analytic extension $F$ to some open connected neighborhood $U$ of $X$.  
    For fixed $v \in \C^d$, the real analytic map $g(x) := i F'(x)(v)-F'(x)(iv)$ on $U$ vanishes on $X^\o$, thus on $U$.
    Hence $F$ is holomorphic.
\end{proof}

In this spirit we define 
\[
    \mathcal{AH}(X,E) := \{ f \in \AA(X,E) : f'(x) \text{ is $\C$-linear for all } x \in X^\o\}.  
\]
Then $\mathcal{AH}(X,E)$ is a closed linear subspace of $\AA(X,E)$ and thus a convenient vector space. 

\begin{theorem}
    Let $X_i \subseteq \R^{d_i}$, $i=1,2$, be simple closed fat subanalytic sets. 
    Then the following exponential law holds as bornological isomorphism:
    \[
        \mathcal{AH}(X_1,\mathcal{AH}(X_2,E)) \cong \mathcal{AH}(X_1 \times X_2,E).
    \]
\end{theorem}

\begin{proof}
    By \Cref{t:expAn}, it suffices to check that $\C$-linearity of the respective derivatives is transferred
    which follows from
    \[
        f'(x_1,x_2)(v_1,v_2) = \ev_{x_2} \big((f^\vee)'(x_1)(v_1)\big) + \big((f^\vee)(x_1)\big)'(x_2)(v_2).
    \]
\end{proof}

\section{Ultradifferentiable maps} \label{sec:ultra}

Ultradifferentiable functions form classes of $\cC^\infty$ functions defined by restrictions on the growth of 
the iterated derivatives. They include the real analytic class, Gevrey classes, Denjoy--Carleman classes, 
and Braun--Meise--Taylor classes.
We will focus on the latter since, under certain circumstances, they 
admit analogues of the smooth and real analytic results of \Cref{sec:smooth,sec:analytic}. 
For background on ultradifferentiable classes, see the survey article \cite{Rainer:2021aa}.

\subsection{Weight functions} \label{ssec:wf}

A \emph{weight function} is a continuous increasing function $\om : [0,\infty) \to [0,\infty)$ with $\om(0)=0$ and satisfying:
\begin{enumerate}
    \item $\om(2t) = O(\om(t))$ as $t \to \infty$;
    \item $\log t = o(\om(t))$ as $t \to \infty$;
    \item $\vh := \om \o \exp$ is convex on $[0,\infty)$.
\end{enumerate}
Two weight functions $\om$ and $\si$ are called \emph{equivalent} if $\om(t) = O(\si(t))$ and $\si(t)=O(\om(t))$ as $t \to \infty$.
Up to equivalence, we may always assume that $\om|_{[0,1]}=0$.
Let $\vh^*$ be the \emph{Young conjugate} of $\vh$,
\[
    \vh^*(t) = \sup_{s \ge 0} (st-\vh(s)),\quad t \ge 0.
\]
We associate with $\om$ a family $\{W^{[\xi]}\}_{\xi>0}$ of positive sequences:
\[
    W^{[\xi]}_k := \exp(\tfrac{1}{\xi} \vh^*(\xi k)), \quad k \in \N.
\]
We will also use $w^{[\xi]}_k := W^{[\xi]}_k/k!$.

\begin{lemma}[{\cite[Lemma 11.3]{Rainer:2021aa} or \cite{RainerSchindl12}}] \label[l]{l:prop}
    Let $\om$ be a weight function with associated family $\{W^{[\xi]}\}_{\xi>0}$.
    Then:
    \begin{enumerate}
        \item Each $W^{[\xi]}$ is a \emph{weight sequence}, i.e., $W^{[\xi]}$ is log-convex and satisfies $W^{[\xi]}_0 =1 \le W^{[\xi]}_1$ and $(W^{[\xi]}_k)^{1/k} \to \infty$ as $k \to \infty$. 
            In particular, $W^{[\xi]}$ is increasing.
        \item $W^{[\xi]}_k \le W^{[\ze]}_k$ for all $k$ if $\xi \le \ze$. 
        \item For all $\xi>0$ and all $j,k \in \N$, $W^{[\xi]}_{j+k} \le W^{[2\xi]}_jW^{[2\xi]}_k$.
        \item $\forall \rh>0 ~\exists H \ge 1 ~\forall \xi >0 ~\exists C \ge 1 ~\forall k \in \N : \rh^k W^{[\xi]}_k \le C W^{[H\xi]}_k$.
    \end{enumerate}
    It is evident, that \thetag{2}, \thetag{3}, and \thetag{4} also hold for the sequences $w^{[\xi]}$ instead of $W^{[\xi]}$.
\end{lemma}

A property of a weight function, that is crucial for this paper, is introduced in the following definition.

\begin{definition}
    A weight function $\om$ is said to be \emph{robust} if it satisfies
    \begin{equation} \label{e:robust}
        \E B> 1 \A t \ge 0: \om(t^2) \le B \om(t) + B.
    \end{equation}
\end{definition}

\begin{lemma}
    Let $\om$ be a robust weight function with associated family $\{W^{[\xi]}\}_{\xi>0}$.
    Then 
    \begin{equation} \label{e:robust2}
        W^{[\xi]}_{2k} \le e^{1/\xi} W^{[B\xi]}_k, \quad k \in \N,\, \xi >0.
    \end{equation}
    Moreover,
    for every integer $a\ge 2$ there are constants $G,H> 1$ such
    \begin{equation} \label{e:robust3}
        W^{[\xi]}_{ak} \le e^{G/\xi} W^{[H\xi]}_k, \quad k \in \N,\, \xi >0.
    \end{equation}
    We may take $G:= \tfrac{B}{B-1}$ and $H := B^p$, 
    if $p$ is an integer with $a \le 2^p$.
\end{lemma}

\begin{proof}
    By \eqref{e:robust}, for all $t\ge 0$,
    \begin{align*}
        B \vh^*(\tfrac{2t}{B}) &= \sup_{s\ge 0} (2s t - B\vh(s)) = \sup_{u\ge 0} (2 t \log u - B\om(u)) 
        \\
                               &\le \sup_{u\ge 0} (t \log u^2 - \om(u^2)) +B = \sup_{v\ge 0} (t \log v - \om(v)) +B 
                               = \vh^*(t) +B. 
    \end{align*}
    Consequently,
    \begin{align*}
        W^{[\xi]}_{2k} = \exp(\tfrac{1}{\xi} \vh^*(2\xi k)) 
        &= \exp(\tfrac{B}{B\xi} \vh^*(\tfrac{2 B\xi k}{B})) 
        \\
        &\le \exp(\tfrac{1}{B\xi} (\vh^*(B\xi k)+B)) = e^{1/\xi} W^{[B\xi]}_k.
    \end{align*}
    Let $p\ge 1$ be an integer such that $a \le 2^p$. 
    Then \eqref{e:robust3} follows by iterating \eqref{e:robust2}:
    \begin{align*}
        W^{[\xi]}_{ak} \le W^{[\xi]}_{2^pk} \le e^{\tfrac{1}{\xi}} W^{[B\xi]}_{2^{p-1}k} 
        \le e^{\tfrac{1}{\xi}\big(1+\tfrac{1}{B}\big)} W^{[B^2\xi]}_{2^{p-2}k} \le 
        \cdots \le e^{\tfrac{1}{\xi} \sum_{i=0}^{p-1}\tfrac{1}{B^i}}  W^{[B^p\xi]}_{k}.
    \end{align*}
\end{proof}

\begin{lemma} \label[l]{l:robust>}
    Any robust weight function $\om$ is \emph{strong}, i.e.,
    \begin{equation} \label{eq:strong}
        \exists C >0 ~\A t > 0 : \int_1^\infty \frac{\om(tu)}{u^2}\, du \le C\om(t) +C,
    \end{equation}
    and equivalent to a concave weight function. 
\end{lemma}

\begin{proof}
    By \cite[Proposition 1.3]{MeiseTaylor88}, 
    a weight function $\om$ satisfies \eqref{eq:strong} if and only if there exists a constant $K>1$ such that 
    \[
        \limsup_{t\to \infty} \frac{\om(Kt)}{\om(t)} < K.
    \]
    Clearly, \eqref{e:robust} implies this condition. 

    It is furthermore proved in \cite[Proposition 1.3]{MeiseTaylor88} that 
    a strong weight function $\om$ is equivalent 
    to the concave weight function $\ka(t) := \int_1^\infty \frac{\om(tu)}{u^2}\, du$. 
\end{proof}

Any strong weight function $\om$ is \emph{non-quasianalytic}, meaning that
\[
   \int_1^\infty \frac{\om(t)}{t^2}\, dt < \infty. 
\]
If this integral is divergent, then $\om$ is called \emph{quasianalytic}.
We remark that, if $\om$ is non-quasianalytic, then each associated weight sequence $W^{[\xi]}$ is non-quasianalytic, 
i.e., $\sum_{k} (W^{[\xi]}_k)^{-1/k} < \infty$; see e.g.\ \cite[Theorem 11.17]{Rainer:2021aa} and \Cref{ssec:ws}.

\begin{example}
    For each $s>1$, $\om_s(t) := ((\log t)_+)^s$ is a robust weight function. 
    On the other hand, $\ga_s(t) := t^{1/s}$, for $s>1$, are strong weight functions that are not robust; 
    they give rise to the Gevrey classes $\cG^s = \cE^{\{\ga_s\}}$.
\end{example}

\subsection{The ultradifferentiable classes \texorpdfstring{$\cE^{\{\om\}}(X)$}{Eomega(X)}} \label{d:Eom}
Let $\om$ be a weight function.
Let $X \subseteq \R^d$ be nonempty and either open or closed. Let $\cE^{\{\om\}}(X)$ be the set of all functions $f \in \cC^\infty(X)$ 
such that for all compact $K \subseteq X$ there exists $\rh>0$ such that 
\begin{equation}
    \|f\|^\om_{K,\rh} := \sup_{x \in K}\sup_{\al \in \N^d} |\p^\al f(x)|\exp(-\tfrac{1}{\rh} \vh^*(\rh |\al|)) < \infty.
\end{equation}
Equivalently, $f \in \cE^{\{\om\}}(X)$ if and only if $f \in \cC^\infty(X)$ and
for all compact $K \subseteq X$ there exist $\xi,\rh>0$ such that 
\begin{equation}
    \sup_{x \in K}\sup_{\al \in \N^d} \frac{|\p^\al f(x)|}{\rh^{|\al|} W^{[\xi]}_{|\al|}} < \infty.
\end{equation}
This follows from \Cref{l:prop}(4); see \cite[Theorem 11.4]{Rainer:2021aa} and \cite[Theorem 5.14]{RainerSchindl12}.

We will be interested in the case that 
$X \subseteq \R^d$ is a simple fat closed subanalytic set. 
We topologize $\cE^{\{\om\}}(X)$ by 
\[
    \cE^{\{\om\}}(X) = \on{proj}_{n \in \N} \on{ind}_{m \in \N} \cE^\om_m(X \cap \ol B(0,n)),
\]
where 
\[
    \cE^\om_m(X \cap \ol B(0,n)) := \{ f \in \cC^\infty(X \cap \ol B(0,n)) : \| f\|^\om_{X \cap \ol B(0,n),m} < \infty\}
\]
is a Banach space, by Whitney's extension theorem.
The inductive limit can be equivalently written as an inductive limit with compact connecting mappings; see \Cref{l:compact}. 
It follows that $\cE^{\{\om\}}(X)$ is complete and webbed and hence 
satisfies the uniform boundedness principle with respect to point evaluations.

\begin{lemma} \label[l]{l:ubEom}
    Let $X \subseteq \R^d$ be a simple fat closed subanalytic set.
    Then $\cE^{\{\om\}}(X)$ satisfies the uniform $\cS$-boundedness principle for $\cS=\{\on{ev}_x : x \in X\}$.
\end{lemma}

The topology on $\cE^{\{\om\}}(X)$ is defined analogously if $X \subseteq \R^d$ is an open set (see e.g.\ \cite{BMT90}). 
In particular, this gives the topology on $\cE^{\{\om\}}(\R)$ that will be used below.

Let us recall some facts; details can be found in \cite{Rainer:2021aa}. 
The class $\cE^{\{\om\}}$ is non-quasianalytic and hence admits nontrivial functions with compact support if and only if $\om$ is non-quasianalytic.
It is stable under composition if and only if $\om$ is equivalent to a concave weight function (see \cite{FernandezGalbis06} and \cite{RainerSchindl12}).

If $\om$ is robust and $X\subseteq \R^d$ is a simple fat closed subanalytic set,
then each function in $\cE^{\{\om\}}(X)$ extends to a function in $\cE^{\{\om\}}(\R^d)$. Indeed, the strong weight functions are precisely those among the non-quasianalytic 
ones that admit a $\cE^{\{\om\}}$ version of the Whitney extension theorem, by \cite{BBMT91}.
That a function in $\cE^{\{\om\}}(X)$ defines a Whitney jet of class $\cE^{\{\om\}}$ on $X$
follows from \eqref{e:robust3} and \cite[Lemma 10.1]{Rainer18}.

\subsection{Weight sequences} \label{ssec:ws} 

By definition, a \emph{weight sequence} is a positive log-convex sequence $M=(M_k)$ satisfying $M_0=1\le M_1$ and $(M_k)^{1/k} \to \infty$ (see \Cref{l:prop}(1)). 
Log-convexity means that the sequence $\mu_k = M_k/M_{k-1}$ is increasing. 
Thus, a weight sequence is increasing and also $(M_k)^{1/k}$ is increasing.
Note that $(M_k)^{1/k} \le \mu_k$ and  
$(M_k)^{1/k} \to \infty$ if and only if $\mu_k \to \infty$; see \cite[Lemma 2.3]{Rainer:2021aa}.
A weight sequence $M=(M_k)$ is \emph{non-quasianalytic} if 
\[
    \sum_k \frac{1}{(M_k)^{1/k}}< \infty \quad \text{ or equivalently } \quad \sum_k \frac{1}{\mu_k} < \infty. 
\]

For later reference, we recall the definition of curves of class $\cE^{\{M\}}$ in a Banach space $E$: 
\[
    \cE^{\{M\}}(\R,E) := \Big\{ f \in \cC^\infty(\R,E) : \forall r>0 ~\exists \rh>0 : \sup_{|t|\le r} \sup_{k \in \N} \frac{\|f^{(k)}(t)\|}{\rh^k M_k} < \infty \Big\}.
\]

\subsection{The space \texorpdfstring{$\AU(X)$}{AEomega}}

Let $X \subseteq \R^d$ be nonempty. We consider the vector space of \emph{arc-$\cE^{\{\om\}}$ functions},
\[
    \AU(X) := \{f : X \to \R : f_* \cE^{\{\om\}}(\R,X) \subseteq \cE^{\{\om\}}(\R,\R)\}
\]
and equip it with the initial locally convex structure with respect to the family of maps 
\[
    \AU(X) \stackrel{c^*}{\longrightarrow}    \cE^{\{\om\}}(\R,\R), \quad c \in \cE^{\{\om\}}(\R,X),
\]
where $\cE^{\{\om\}}(\R,\R)$ carries the topology described in \Cref{d:Eom}.
Then the space $\AU(X)$ is $c^\infty$-closed in the product $\prod_{c \in \cE^{\{\om\}}(\R,X)} \cE^{\{\om\}}(\R,\R)$ 
and thus a convenient vector space.

\begin{lemma} \label[l]{l:ubAom}
    $\AU(X)$ satisfies the uniform $\cS$-boundedness principle for $\cS=\{\on{ev}_{x} : x \in X\}$.
\end{lemma}

\begin{proof}
    This follows from \Cref{l:stab} and the fact that $\cE^{\{\om\}}(\R,\R)$ is webbed.
\end{proof}

\begin{remark}
    (1) Let $\om$ be a quasianalytic concave weight function satisfying $\om(t) =o(t)$ as $t \to \infty$ (i.e., the quasianalytic class $\cE^{\{\om\}}$ strictly contains 
    the real analytic class).
    Then, by \cite[Theorem 3]{Rainer:2019aa}, for each integer $d \ge 2$ and each positive sequence $N=(N_k)$ there exists 
    $f \in \AU(\R^d) \cap \cC^\infty(\R^d)$ such that $f|_{\R^d \setminus \{0\}} \in \cE^{\{\om\}}(\R^d \setminus \{0\})$ 
    but, for all $r,\rh>0$,
    \[
        \sup_{x \in [-r,r]^d} \sup_{\al \in \N^d}  \frac{|\p^\al f(x)|}{\rh^{|\al|} N_{|\al|}} =\infty.
    \]
    Thus we will only consider non-quasianalytic weight functions $\om$.

    (2)
    The function $f: X:=\{(x,y) \in \R^2 : x^3=y^2\} \ni (x,y) \mapsto y^{1/3} \in \R$ belongs to $\AU(X)$ 
    for any concave weight function $\om$ (quasianalytic or non-quasianalytic), 
    which follows from \cite[Corollary 1.2]{Nenning:2021tv} (see also \cite{Nenning:2021wd} and \cite{Thilliez:2020ac}), 
    but clearly $f \not\in \cC^\infty(X)$. 
\end{remark}

\subsection{\texorpdfstring{$\AU(X)$}{AEomega(X)} and \texorpdfstring{$\cE^{\{\om\}}(X)$}{Eomega(X)} coincide for 
robust \texorpdfstring{$\om$}{omega} and suitable \texorpdfstring{$X$}{X}}

The following lemma goes back to \cite{Boman67}; we recall a version that appeared in \cite[Lemma 2.4]{Rainer18} and 
repeat the proof for later reference.

We start with the setup for the next lemma.
Choose a sequence 
\begin{equation} \label{eq:Tt}
    \text{$T_j \in (0,1]$ with $\sum_j T_j < \infty$ and let $t_k := 2 \sum_{j<k} T_j + T_k$.}
\end{equation}
Then $t_k \to t_\infty \in \R$.

Let $M=(M_k)$ be any non-quasianalytic weight sequence.
There exists a non-quasianalytic weight sequence $L = (L_k)$ such that $(M_k/L_k)^{1/k} \to \infty$
(use e.g.\ \cite[Corollary 3.5]{Rainer:2021aa} with $\al_k=1/\mu_k$).

Choose a decreasing sequence $\la_j>0$ such that the following conditions are fulfilled: 
\begin{align}
        &0 < \frac{\la_j}{T_j^k} \le \frac{M_k}{L_k} \quad \text{ for all }  j,k, \label{la1}
        \\
        &\frac{\la_j}{T_j^k} \to 0 \quad \text{ as }  j \to \infty \text{ for all } k. \label{la2}
\end{align}
It suffices to take $\la_j \le \inf_{k} T_j^{k+1} M_k/L_k$.

\begin{lemma}[{\cite[Lemma 2.4]{Rainer18}}] \label[l]{l:curve}
    Let $(T_k)$ and $(t_k)$ be the sequences defined in \eqref{eq:Tt}.
    Let $M=(M_k)$ and $L=(L_k)$ be non-quasianalytic weight sequences such that $(M_k/L_k)^{1/k} \to \infty$.
    Let $(\la_k)$ be a positive sequence satisfying \eqref{la1} and \eqref{la2}.    

    If $(c_k)$ is a sequence in $\cC^\infty(\R,E)$, where $E$ is a Banach space, such that the set 
    \begin{equation} \label{ass}
        \Big\{ \frac{c_k^{(\ell)}(t)}{\la_k} : t \in I, \, k,\ell \in \N \Big\}
    \end{equation}
    is bounded in $E$ for each bounded interval $I \subseteq \R$,
    then there exists $c \in \cE^{\{M\}}(\R, E)$ with compact support 
    such that $c(t_k+t) = c_k(t)$ for $|t| \le T_k/3$.
\end{lemma}

\begin{proof}
    Choose a $\cE^{\{L\}}$-function $\vh: \R \to [0,1]$ which is $0$ on $\{t : |t| \ge 1/2\}$ and $1$ 
    on $\{t : |t| \le 1/3\}$.   
    Define 
    \[
        c(t) := \sum_j \vh\Big(\frac{t-t_j}{T_j}\Big) c_j(t-t_j).
    \]
    The summands have disjoint supports. 
    Thus $c$ is $\cC^\infty$ on $\R \setminus \{t_\infty\}$. 
    By assumption \eqref{ass}, there is $R >0$ such that   
    \[
        \|c_k^{(\ell)}(t)\| \le R \la_k \quad \text{ for all }  |t|\le 1/2,\, \ell,k \in \N.
    \] 
    So there exist $C,\rh \ge 1$ such that, for $|t-t_j|\le T_j/2$,
    \begin{align*} 
        \|c^{(\ell)}(t)\| &= \Big\|\sum_{i=0}^\ell \binom{\ell}{i} T_j^{-i} \vh^{(i)}\Big(\frac {t-t_j}{T_j}\Big) c_j^{(\ell-i)}(t-t_j) \Big\|
        \\
                          &\le R\la_j \sum_{i=0}^\ell  \binom{\ell}{i} T_j^{-i} C \rh^{i}  L_{i} \notag
                          \le C R\la_j  \Big(1+\frac{\rh}{T_j}\Big)^{\ell}  L_{\ell}
                          \le C R\la_j  \Big(\frac{2\rh}{T_j}\Big)^{\ell}  L_{\ell}   .
    \end{align*}
    Consequently, by \eqref{la1},
    \[
        \|c^{(\ell)}(t)\| \le CR   (2\rh)^{\ell}   M_{\ell}     
        \quad \text{ for } t \ne t_\infty. 
    \]
    It follows that $c : \R \to E$ has compact support and is of class $\cE^{\{M\}}$ (cf.\ \cite[Lemma 2.9]{KM97} and \cite[Lemma 3.7]{KMRc}).
\end{proof}

\begin{lemma} \label[l]{l:omoo}
    Let $\om$ be a non-quasianalytic weight function.
    Let $X \subseteq \R^d$ be a H\"older set or a simple fat closed subanalytic set. 
    Then we have the bounded 
    inclusion
    \[ 
        \AU(X) \subseteq \AS(X) = \cC^\infty(X).
    \]
\end{lemma}

\begin{proof}
    We show the bounded inclusion 
    $\AU(X) \subseteq \cC^\infty(X)$; the rest was seen in \Cref{t:smooth1}.

    Let $\{W^{[\xi]}\}_{\xi>0}$ be the family of weight sequences 
    associated with $\om$ and fix $\xi_0>0$. 
    Let $f \in \AU(X)$ and $c \in \cE^{\{W^{[\xi_0]}\}}(\R,X)$.
    Then $f \o c \in \cC^\infty(\R,\R)$ since $\cE^{\{W^{[\xi_0]}\}}(\R,X) \subseteq \cE^{\{\om\}}(\R,X)$. 
    This implies that $f \in \cC^\infty(X)$ by \cite[Theorem 1.13]{Rainer18} if $X$ is a H\"older set.
    (In \cite{Rainer18} it was assumed that the sequence $(M_k/k!)$ is log-convex but this assumption 
    is not needed in the proof; cf.\ \Cref{l:curve}.)
    Similarly, we get that $f \in \cC^\infty(X)$ if $X$ is a simple closed fat subanalytic set:
    repeat the proof of \cite[Theorem 1.14]{Rainer18}, notice that 
    the composites of $\cE^{\{W^{[\xi_0]}\}}$ curves  with the polynomial maps $\Ps_{x,v}$ are 
    still $\cE^{\{W^{[\xi_0]}\}}$ curves (because $W^{[\xi_0]}$ being log-convex implies 
    the ring property), and use \cite[Lemma 2.6]{Rainer18} in the argument for Claim 1.

    That the inclusion is bounded follows from \Cref{l:ubA} or \Cref{l:ubC}.
\end{proof}

\begin{theorem} \label[t]{t:Lip}
    Let $\om$ be a robust weight function.
    For each Lipschitz set $X \subseteq \R^d$,  
    \begin{equation} \label{e:AUL}
        \AU(X) = \cE^{\{\om\}}(X). 
    \end{equation}
\end{theorem}

\begin{proof}
    The inclusion $\cE^{\{\om\}}(X) \subseteq \AU(X)$ is an easy consequence of Fa\`a di Bruno's formula;
    the fact that $\om$ is (up to equivalence) concave, by \Cref{l:robust>}, 
    entails that the class $\cE^{\{\om\}}$ is stable under composition, 
    see \cite{Rainer:2021aa} or \cite{RainerSchindl12}.

    To prove the inclusion $\AU(X) \subseteq \cE^{\{\om\}}(X)$,
    let $f \in \AU(X)$. Then $f \in \cC^\infty(X)$, by \Cref{l:omoo}.
    Suppose for contradiction that $f \not\in \cE^{\{\om\}}(X)$. 
    Then, in view of \cite[Proposition 7.2]{Rainer18}, there exists $a \in X$ such that for all 
    $\de,C,\xi, \rh>0$ and all nonempty open subsets $V$ of $\mathbb S^{d-1}$  
    there exist $x \in X \cap B(a,\de)$, $v \in V$, and $k \in \N$
    such that
    \begin{equation} \label{e:contra}
        |d_v^k f(x)| > C \rh^k W^{[\xi]}_k. 
    \end{equation}
    We may assume that $a \in \p X$ because $\AU(X^\o) = \cE^{\{\om\}}(X^\o)$ (see \cite[Theorem 3.9]{KMRc} 
    the proof of which can be adapted easily to the case $\cE^{\{\om\}}$).
    Then, $X$ being a Lipschitz set, there is $\ep>0$ and a truncated open cone $\Ga$ 
    such that $y+\Ga \subseteq X^\o$ for all $y \in X \cap B(a,\ep)$ (in suitable coordinates).
    Set $C(y,r):= y + r \Ga$, for $0<r\le 1$.
    There is a universal constant $c>0$ such that $C(y_1,r_1)\cap C(y_2,r_2) \ne \emptyset$ if $|y_1-y_2| < c \min\{r_1,r_2\}$.

    Fix $\xi_0>0$  
    and a non-quasianalytic weight sequence $L$ satisfying  $(W^{[\xi_0]}_k/L_k)^{1/k} \to \infty$.
    Let $(T_n)$ and $(t_n)$ be the sequences defined in \eqref{eq:Tt} and
    $(\la_k)$ a sequence satisfying \eqref{la1} and \eqref{la2} for $M=W^{[\xi_0]}$.

    Let $B\ge 1$ be the constant from \eqref{e:robust}. 
    Taking $\de:= c \la_{n+1}/3$, $C := e^{1/n}$, $\xi := Bn$, $\rh := \la_n^{-3}$,
    and $V := \mathbb S^{d-1} \cap \R_+ \Ga$,
    we find sequences $x_n \in X \cap B(a,c\la_{n+1}/3)$, $v_n \in \mathbb S^{d-1} \cap \R_+ \Ga$, and $k_n \in \N$ such that
    \begin{equation} \label{e:contra2}
        |d_{v_n}^{k_n}f(x_n)| \ge e^{1/n} \la_n^{-3k_n} W^{[Bn]}_{k_n}\quad \text{ for all } n.
    \end{equation}

    Consider $C_n:= C(x_n,\la_n)$ which lies in $X^\o$ for sufficiently large $n$.
    Since $|x_n -x_{n+1}| < c \la_{n+1}$ there is a sequence $(u_n)$ of points in $X$ 
    satisfying $u_{n+1} \in C_n \cap C_{n+1}$ for all $n$.
    Note that $x_n$ and $u_n$ are $\la_n^{-1}$-convergent to $a$.

    After a translation, we may assume that $a=0$.
    Consider the curves  $c_n(t) = x_n + t^2 \la_n v_n$. 
    Choose a $\cE^{\{L\}}$-function $\vh: \R \to [0,1]$ which is $0$ on $\{t : |t| \ge 1/2\}$ and $1$ 
    on $\{t : |t| \le 1/3\}$.  

    For $t \in [t_n-T_n,t_n+T_n]$, we set
    \[
        c(t):= \vh(\tfrac{t-t_n}{T_n}) c_n(t-t_n) + (1-\vh(\tfrac{t-t_n}{T_n})) (u_n \mathbf 1_{(-\infty,t_n]}(t) 
        + u_{n+1} \mathbf 1_{[t_n,+\infty)}(t))
    \]
    and $c(t):=0$ for $t \in [t_\infty,\infty)$.
    By the proof of \Cref{l:curve}, $c$ is a curve of class $\cE^{\{W^{[\xi_0]}\}}$, in particular, of class $\cE^{\{\om\}}$,
    which lies in $X$, by construction.

    Now, for all $k$,
    \[
        (f \o c)^{(2k)}(t_n) = \frac{(2k)!}{k!} \la_n^k d_{v_n}^k f(x_n);
    \]
    here we use that $f \in \cC^\infty(X)$ so that the use of the chain rule is justified.
    By \eqref{e:robust2} and \eqref{e:contra2}, we conclude
    \begin{align*}
        \Big(\frac{|(f \o c)^{(2k_n)}(t_n)|}{W^{[n]}_{2k_n}}\Big)^{\frac{1}{2k_n}} 
         &= \Big(\frac{(2k_n)!\,\la_n^{k_n} |d_{v_n}^{k_n} f(x_n)|}{k_n!\,W^{[n]}_{2k_n}}\Big)^{\frac{1}{2k_n}}
         \\
         &\ge \Big(\frac{(2k_n)!}{k_n!}\frac{\la_n^{k_n} |d_{v_n}^{k_n} f(x_n)|}{ e^{1/n} W^{[Bn]}_{k_n}}\Big)^{\frac{1}{2k_n}}
         \ge \frac{1}{\la_n} \to \infty. 
    \end{align*}
    This contradicts the assumption $f \in \AU(X)$.
\end{proof}

\begin{lemma} \label[l]{l:pull}
    Let $\om$ be a robust weight function.
    Let $X_i \subseteq \R^{d_i}$, $i=1,2$, and $\vh \in \cE^{\{\om\}}(X_1,X_2)$, i.e., all components $\on{pr}_j \o \vh$ of the map $\vh : X_1 \to X_2$ belong to $\cE^{\{\om\}}(X_1)$.
    If $\AU(X_1) = \cE^{\{\om\}}(X_1)$, then $\vh^* \AU(X_2) \subseteq \cE^{\{\om\}}(X_1)$.
\end{lemma}

\begin{proof}
    Let $f \in \AU(X_2)$. Assume for contradiction that $f \o \vh \not \in \cE^{\{\om\}}(X_1) = \AU(X_1)$. 
    Thus there exists a $\cE^{\{\om\}}$ curve $c$ in $X_1$ such that $f\o \vh \o c \not \in \cE^{\{\om\}}(\R,\R)$. 
    Since $\vh\o c$ is a $\cE^{\{\om\}}$ curve in $X_2$ (because $\cE^{\{\om\}}$ is stable under composition) 
    this contradicts $f \in \AU(X_2)$.
\end{proof}

\begin{proposition} \label[p]{p:rect}
    Let $\om$ be a robust weight function.
    Let $X \subseteq \R^d$ be a fat closed subanalytic set.
    There is a locally finite collection of real analytic maps $\vh_\al : U_\al \to \R^d$, 
    where the $U_\al$ are open sets in $\R^d$, such that 
    \[
        \vh_\al^* \AU(X) \subseteq \cE^{\{\om\}}(\vh_\al^{-1}(X)) \quad \text{ for all } \al. 
    \]
\end{proposition}

\begin{proof}
    We use the rectilinearization theorem for subanalytic sets (see \cite{BM88}).
    There exists a locally finite collection of real analytic maps $\vh_\al : U_\al \to \R^d$ 
    such that each $\vh_\al$ is the composite of a finite sequence of local blowings-up with smooth centers and 
    \begin{itemize}
        \item each $U_\al$ is diffeomorphic to $\R^d$ and there are compact subsets $K_\al \subseteq U_\al$ such that $\bigcup_\al \vh_\al(K_\al)$ 
            is a neighborhood of $X$ in $\R^d$,
        \item $\vh_\al^{-1}(X)$ is a union of quadrants in $\R^d$, for each $\al$.
    \end{itemize}
    A quadrant is a set 
    \[
        Q(I_0,I_-,I_+)= \{x \in \R^d : x_i=0 \text{ if } i \in I_0,\, x_i \le 0 \text{ if } i \in I_-,\,  x_i \ge 0 \text{ if } i \in I_+\},
    \]
    where $I_0,I_-,I_+$ is any partition of $\{1,2,\ldots,d\}$.
    In our case, $I_0 = \emptyset$ since $X$ is fat.

    We claim that for any union $Y$ of quadrants $Q(\emptyset,I_-,I_+)$ we have $\AU(Y) = \cE^{\{\om\}}(Y)$. 
    Then \Cref{l:pull} implies the assertion of the proposition.

    To see the claim, let $f \in \AU(Y)$. Then $f \in \cC^\infty(Y)$, by (the proof of) \cite[Theorem 8.2]{Rainer18}. 
    Hence it suffices to check that $f$ satisfies the defining estimates on each compact subset of $Y$ (see \Cref{d:Eom}).
    This follows from the fact that, by \Cref{t:Lip}, the estimates hold on all compact subsets of each of the finitely many quadrants that make up $Y$. 
    The inclusion $\cE^{\{\om\}}(Y) \subseteq \AU(Y)$ is a simple consequence 
    of the fact that $\cE^{\{\om\}}$ is stable under composition.
\end{proof}

\begin{theorem} \label[t]{t:ultra}
    Let $\om$ be a robust weight function.
    Let $X \subseteq \R^d$ be a simple fat closed subanalytic set.
    Then 
    \begin{equation} \label{e:AUS}
        \AU(X) = \cE^{\{\om\}}(X)
    \end{equation}
    and the identity is a bornological isomorphism.
\end{theorem}

\begin{proof}
    Let us show the inclusion $\AU(X) \subseteq \cE^{\{\om\}}(X)$; the opposite inclusion follows easily by Fa\`a di Bruno's formula.
    To this end we work with the maps $\vh_\al$ from the proof of \Cref{p:rect}.
    We may assume that the Jacobian determinant of each $\vh_\al$ is a monomial times a nowhere vanishing factor.
    Let $f \in \AU(X)$. By \Cref{l:omoo}, we have $f \in \cC^\infty(X)$.
    By \Cref{p:rect}, $f \o \vh_\al \in \cE^{\{\om\}}(Y_\al)$, where $Y_\al$ is a union of quadrants $Q(\emptyset,I_-,I_+)$ in $\R^d$.
    By \Cref{t:A1},
    $f$ is of class $\cE^{\{\om\}}$ on $\vh_\al(Y_\al)$ for all $\al$.
    We conclude that $f \in \cE^{\{\om\}}(X)$.

    The identity \eqref{e:AUS} is a bornological isomorphism, by \Cref{l:ubEom} and \Cref{l:ubAom}.
\end{proof}

\subsection{The vector valued case}

Let $X\subseteq \R^d$ be nonempty and $E$ a convenient vector space.
We consider the set
$\AU(X,E)$ 
of all maps $f : X \to E$ such that $\ell \o f \o c \in \cE^{\{\om\}}(\R,\R)$ 
for all $c \in \cE^{\{\om\}}(\R,X)$ and $\ell \in E^*$.
We equip $\AU(X,E)$ with the initial locally convex structure with respect to the family of maps 
\[
    \AU(X,E) \stackrel{\ell_* \o c^*}{\longrightarrow} \cE^{\{\om\}}(\R,\R), \quad c\in \cE^{\{\om\}}(\R,X), ~ \ell \in E^*,
\]
which makes $\AU(X,E)$ a convenient vector space. 

\begin{lemma} \label[l]{l:ubAomE}
    $\AU(X,E)$ satisfies the uniform $\cS$-boundedness principle for $\cS=\{\on{ev}_{x} : x \in X\}$.
\end{lemma}

\begin{proof}
    This follows from \Cref{l:ubEom} (or \Cref{l:ubAom}) and \Cref{l:stab}.
\end{proof}

Let $X \subseteq \R^d$ be a simple fat closed subanalytic set. 
We define the set
\[
    \cE^{\{\om\}}(X,E) := \{ f : X \to E : \ell \o f \in \cE^{\{\om\}}(X) \text{ for all } \ell \in E^*\}
\]
and endow it with the initial locally convex structure with respect to the family of maps 
\[
    \cE^{\{\om\}}(X,E) \stackrel{\ell_*}{\longrightarrow} \cE^{\{\om\}}(X), \quad \ell \in E^*.
\]
Then $\cE^{\{\om\}}(X,E)$ is a convenient vector space. 

\begin{lemma} \label[l]{l:ubComE}
    Let $X \subseteq \R^d$ be a simple fat closed subanalytic set.
    Then $\cE^{\{\om\}}(X,E)$ satisfies the uniform $\cS$-boundedness principle for $\cS=\{\on{ev}_{x} : x \in X\}$.
\end{lemma}

\begin{proof}
    This follows from \Cref{l:ubEom} and \Cref{l:stab}.
\end{proof}

\begin{theorem} \label[t]{t:ultra2}
    Let $\om$ be a robust weight function.
    Let $X \subseteq \R^d$ be a simple closed fat subanalytic set and $E$ a convenient vector space. Then
    \[
        \AU(X,E) = \cE^{\{\om\}}(X,E)
    \]
    and the identity is a bornological isomorphism.
\end{theorem}

\begin{proof}
    The set-theoretic identity follows from \Cref{t:ultra}, by composing with $\ell \in E^*$.  
    It is a bornological isomorphism, by \Cref{l:ubAomE} and \Cref{l:ubComE}. 
\end{proof}

\subsection{Exponential laws}

\begin{theorem}
    Let $\om$ be a robust weight function.
    Let $X_i \subseteq \R^{d_i}$, $i=1,2$, 
    be simple closed fat subanalytic sets.
    Let $E$ be a convenient vector space.
    Then the following exponential laws hold as bornological isomorphisms:
    \begin{enumerate}
        \item  $\AU(X_1,\AU(X_2,E)) \cong \AU(X_1 \times X_2,E)$;
        \item $\cE^{\{\om\}}(X_1,\cE^{\{\om\}}(X_2,E)) \cong \cE^{\{\om\}}(X_1 \times X_2,E)$.
    \end{enumerate}
\end{theorem}

\begin{proof}
    (1) It is well-known that (1) holds for the special case $X_1=X_2 =E=\R$; see \cite{Schindl14a}.
    Let $f \in \AU(X_1,\AU(X_2,E))$. 
    Let $c_i \in \cE^{\{\om\}}(\R,X_i)$, $i=1,2$, and $\ell \in E^*$. 
    Then 
    \begin{equation} \label{e:red}
        (\ell \o f^\wedge \o (c_1 \times c_2))^\vee = (\ell_* \o c_2^*) \o f \o c_1 : \R \to \cE^{\{\om\}}(\R,\R)
    \end{equation}
    is of class $\cE^{\{\om\}}$. By the special case, $\ell \o f^\wedge \o (c_1 \times c_2) \in \cE^{\{\om\}}(\R^2)$ 
    and thus $f^\wedge \in \AU(X_1\times X_2,E)$.

    Conversely, let $g \in \AU(X_1\times X_2,E)$. Then $g^\vee$ takes values in $\AU(X_2,E)$.
    Each continuous linear functional on $\AU(X_2,E)$ factors over 
    some map $\ell_* \o c_2^* : \AU(X_2,E) \to \cE^{\{\om\}}(\R,\R)$ 
    for some $c_2 \in \cE^{\{\om\}}(\R,X_2)$ and $\ell \in E^*$.
    So it suffices to show that, for each $c_1 \in \cE^{\{\om\}}(\R,X_1)$, the map \eqref{e:red} 
    with $f=g^\vee$ is of class $\cE^{\{\om\}}$. 
    This follows from the special case.

    Thus we have proved that $f \in \AU(X_1,\AU(X_2,E))$ if and only if $f^\wedge \in \AU(X_1 \times X_2,E)$.
    To see that it is a bornological isomorphism we may proceed precisely as in the proof of 
    \Cref{t:expAn}(1), using the uniform boundedness principle \Cref{l:ubAomE}.
    Alternatively, it follows from the fact that \Cref{r:can} applies to the present situation.

    (2) This follows from (1), \Cref{l:prod}, and \Cref{t:ultra2}, by arguments similar to those 
    in the proof of \Cref{t:exp}(2). 
\end{proof}

\appendix

\section{} \label{sec:appendix}

Let $\om$ be a weight function and $U\subseteq \R^d$ an open set. 
Then $\cE^{\{\om\}}(U)$ is a differential ring with respect to multiplication of functions and 
the partial derivatives $\p_i$, $i=1,\ldots,d$; see e.g.\ \cite{Rainer:2021aa}.
Consequently, $\cE^{\{\om\}}(U)$ is stable by division of coordinates:
if $f \in \cE^{\{\om\}}(U)$ and $f|_{\{x_i=a_i\}} =0$, then $f(x) = (x_i-a_i)g(x)$ for $g \in \cE^{\{\om\}}(U)$.
Indeed, 
\[
    f(x_1,\ldots,x_i,\ldots,x_d) 
    = (x_i-a_i) \int_0^1  \p_i f(x_1,\ldots,a_i+t(x_i-a_i),\ldots,x_d)\, dt.
\]
It is not hard to see that multiplication $m : \cE^{\{\om\}}(U) \times \cE^{\{\om\}}(U) \to  \cE^{\{\om\}}(U)$
and differentiation $\p_i : \cE^{\{\om\}}(U) \to  \cE^{\{\om\}}(U)$ are continuous.

If $\om$ is a concave (up to equivalence) weight function, then $\cE^{\{\om\}}$ has strong stability 
properties. In particular, $\cE^{\{\om\}}$ is stable under composition and taking reciprocals: 
if $f \in \cE^{\{\om\}}(U)$ does nowhere vanish, then $1/f \in \cE^{\{\om\}}(U)$; see \cite{RainerSchindl14}.

Let $\{W^{[\xi]}\}_{\xi>0}$ be the family of weight sequences associated with $\om$ and write 
$w^{[\xi]}_k:=  W^{[\xi]}_k/k!$ (see \Cref{ssec:wf}).
If $\om$ is concave, hence subadditive, then it follows easily from the definition (e.g.\ \cite[Lemma 6.1]{RainerSchindl12}) that, for each $\xi>0$,
\begin{equation} \label{e:subadd}
    w^{[\xi]}_j w^{[\xi]}_k \le w^{[\xi]}_{j+k} \quad \text{ for all } j,k.
\end{equation}

We present a version of \cite[Theorem 1.4]{BelottoBierstoneChow17} adapted to the class $\cE^{\{\om\}}$
for robust weight functions $\om$.

\begin{theorem} \label[t]{t:A1}
    Let $\om$ be a robust weight function.
    Let $\vh : U \to V$ be a $\cE^{\{\om\}}$ map between open subsets of $\R^d$.
    Assume that the Jacobian determinant of $\vh$ is a monomial times a nowhere vanishing factor.
    Let $f \in \cC^\infty(V)$ and assume that $f \o \vh \in \cE^{\{\om\}}(U)$.
    Then, for each compact $K \subseteq U$, $f|_{\vh(K)} \in \cE^{\{\om\}}(\vh(K))$. 

    This holds in a bounded way:
    If $\cB$ is a subset of $\cC^\infty(V)$ and there exist  $\xi>0$ and  $C,\rh \ge 1$ such that
    \begin{equation*}
        \sup_{f \in \cB} \sup_{x \in K} \sup_{\al \in \N^d} \frac{|\p^\al(f \o \vh)(x)|}{\rh^{|\al|} W^{[\xi]}_{|\al|}} \le  C, 
    \end{equation*}
    then there exist $A>0$ and $\si\ge 1$ such that 
    \begin{equation*}
        \sup_{f \in \cB} \sup_{y \in \vh(K)} \sup_{\al \in \N^d} \frac{|\p^\al f)(y)|}{\si^{|\al|} W^{[H\xi]}_{|\al|}} \le  A, 
    \end{equation*}
    and where $H\ge 1$ is a constant depending only on $\om$ and $\vh$.
\end{theorem}

\begin{proof}
    We may assume that $K = [-r,r]^d$, for some $r>0$.
    Let $J(x)$ denote the Jacobian matrix of $\vh(x)$.
    By assumption, $\det J(x) = x^\ga u(x)$, where $\ga \in \N^d$ and $u$ is nowhere vanishing.
    Consider $T(x) := u(x)^{-1} \on{adj} J(x)$, where $\on{adj} J(x)$ is the adjugate matrix of $J(x)$.
    Note that 
    \begin{equation} \label{e:AJ}
        (\p_{x_i}) = J(x) \cdot (\p_{y_j}) \quad \text{ and } \quad (\p_{y_j}) = \frac{T(x)}{x^\ga} \cdot (\p_{x_i}),  
    \end{equation}
    where $(\p_{x_i})$ and $(\p_{y_j})$ are the column vectors of partial derivative operators. 
    By the assumption $\vh \in \cE^{\{\om\}}(U,V)$ and the remarks before the theorem, 
    there exist $\xi_0>0$ and  $C_0,\rh_0 \ge 1$ such that, for all $i,j \in \{1,\ldots,d\}$, $x \in K$, and $\al \in \N^d$,
    \begin{align} \label{e:A0}
        |\p^\al T_{ji}(x)| \le  C_0 \rh_0^{|\al|} W^{[\xi_0]}_{|\al|},
    \end{align}
    where the $T_{ji}$ denote the components of the matrix $T$.
    Since $g := f \o \vh \in \cE^{\{\om\}}(U)$,
    there exist $\xi>0$ and  $C,\rh \ge 1$ such that, for all $x \in K$ and $\al \in \N^d$,
    \begin{align} \label{e:A1}
        |\p^\al g(x)| \le  C \rh^{|\al|} W^{[\xi]}_{|\al|}.
    \end{align}
    It is no restriction to assume that $\xi \ge \xi_0$, $C\ge C_0$, and $\rh \ge \rh_0$. 

    Since $\om$ is robust, we may assume that $\om$ is subadditive, by \Cref{l:robust>}, 
    and consequently \eqref{e:subadd} holds.

    Let $D> d\rh$ be such that $\sum_{\al \in \N^d} (\frac{d\rh}{D})^{|\al|} =: B < \infty$.  
    We claim that, 
    for all $x \in K$ and $\al,\be \in \N^d$, 
    \begin{align} \label{e:Aclaim}
        |\p^\al((\p^\be f) \o \vh)(x)| \le (BCd)^{|\be|+1} D^{(|\ga|+1)|\be|+ |\al|}\, w^{[\xi]}_{(|\ga|+1)|\be|+ |\al|}\, \Ga(\al,\be),
    \end{align}
    where 
    \[
        \Ga(\al,\be) := \al! \prod_{j=1}^{|\be|} \max_{1 \le i \le d} (\al_i + j(\ga_i+1)).
    \]
    Let us proceed by induction on $|\be|$.
    The case $\be =0$ follows from \eqref{e:A1}.
    Fix $\tilde \be \in \N^d$ with $|\tilde \be|>0$. 
    Then $\tilde \be = \be + e_j$ for some $\be \in \N^d$ and some $j \in \{1,\ldots,d\}$.
    Then, by \eqref{e:AJ}, 
    \begin{align*}
        ((\p^{e_j}\p^\be f)\o \vh)(x) &= \sum_{i=1}^d \frac{T_{ji}(x)}{x^\ga}  \p^{e_i}(\p^\be f \o \vh)(x)
        \\
                                      &=\int_{[0,1]^{|\ga|}}\sum_{i=1}^d  \p^\ga \Big( T_{ji} \cdot \p^{e_i}(\p^\be f \o \vh)\Big) (\tilde x)\, Q_0(t) \, dt,
    \end{align*}
    by the fundamental theorem of calculus (applied $|\ga|$ times),
    where 
    \[
        t = (t_{11},\ldots,t_{1\ga_1},t_{21},\ldots,t_{2\ga_2},\ldots,t_{d1},\ldots,t_{d\ga_d}),
    \]
    $\tilde x = (\prod_{\ell=1}^{\ga_k} t_{k\ell} x_k)_{k=1}^d$,
    and $Q_0(t) = \prod_{k=1}^d \prod_{\ell=1}^{\ga_k} t_{k\ell}^{\ga_k-\ell}$.
    Consequently, 
    \begin{align*}
        |\p^\al ((\p^{\tilde \be} f) \o \vh)(x)| &\le   \int_{[0,1]^{|\ga|}}\sum_{i=1}^d  \Big|\p^{\ga+ \al} \Big( T_{ji} \cdot \p^{e_i}(\p^\be f \o \vh)\Big) (\tilde x)  \Big|\, Q_\al(t) \, dt,
    \end{align*}
    where $Q_\al(t) = \prod_{k=1}^d \prod_{\ell=1}^{\ga_k} t_{k\ell}^{\ga_k+\al_k-\ell}$. 
    Now
    \begin{align*}
        \MoveEqLeft \Big|\p^{\ga+ \al} \Big( T_{ji} \cdot \p^{e_i}((\p^\be f) \o \vh)\Big) (\tilde x)  \Big| 
        \\
        &\le \sum_{\ka+\la = \al+\ga} \frac{(\al+\ga)!}{\ka!\la!} |\p^\ka T_{ji}(\tilde x)||\p^{\la+e_i}((\p^\be f) \o \vh) (\tilde x)|
    \end{align*}
    and, by \eqref{e:A0} and the induction hypothesis, 
    \begin{align*}
        \MoveEqLeft |\p^\ka T_{ji}(\tilde x)||\p^{\la+e_i}((\p^\be f) \o \vh) (\tilde x)| 
        \\  
    &\le C   (d\rh)^{|\ka|} \ka!\, w^{[\xi]}_{|\ka|} (BCd)^{|\be|+1}D^{(|\ga|+1)|\be|+ |\la|+1} w^{[\xi]}_{(|\ga|+1)|\be|+ |\la|+1} \Ga(\la+e_i,\be) 
    \\
    &\le C (BCd)^{|\be|+1}D^{(|\ga|+1)|\tilde \be|+ |\al|} \ka!\, w^{[\xi]}_{(|\ga|+1)|\tilde \be|+ |\al|}  \Big(\frac{d\rh}{D}\Big)^{|\ka|} \max_{1\le i \le d} \Ga(\la+e_i,\be),
    \end{align*}
    using \eqref{e:subadd} in the last step.
    Since $\int_{[0,1]^{|\ga|}} Q_\al(t) \, dt = \frac{\al!}{(\al+\ga)!}$, we conclude
    \begin{align*}
        |\p^\al ((\p^{\tilde \be} f) \o \vh)(x)| &\le   B^{|\tilde \be|} C^{|\tilde \be|+1} d^{|\tilde \be|+1}D^{(|\ga|+1)|\tilde \be|+ |\al|} w^{[\xi]}_{(|\ga|+1)|\tilde \be|+ |\al|}
        \\
                                                 & \hspace{1cm} \cdot \sum_{\ka+\la = \al+\ga} \frac{\al!}{\la!} \Big(\frac{d\rh}{D}\Big)^{|\ka|} \max_{1\le i \le d}  \Ga(\la+e_i,\be).
    \end{align*}
    We have (see \cite[pp. 1970--1971]{BelottoBierstoneChow17}) 
    \[
        \frac{\al!}{\la!}\max_{1\le i \le d}  \Ga(\la+e_i,\be) \le \Ga(\al,\tilde \be).
    \]
    So, by the choice of $D$,
    \begin{align*}
        |\p^\al ((\p^{\tilde \be} f) \o \vh)(x)| &\le   (B Cd)^{|\tilde \be|+1}D^{(|\ga|+1)|\tilde \be|+ |\al|} w^{[\xi]}_{(|\ga|+1)|\tilde \be|+ |\al|}\Ga(\al,\tilde \be)
    \end{align*}
    and the claim \eqref{e:Aclaim} is proved.

    Taking $\al = 0$ in \eqref{e:Aclaim} and using \eqref{e:robust3}, we find, for all $y \in \vh(K)$ and all $\be \in \N^d$, 
    \begin{align*}
        |\p^\be f(y)| &\le   (BCd)^{|\be|+1} D^{(|\ga|+1)|\be|} w^{[\xi]}_{(|\ga|+1)|\be|} \Ga(0,\be) 
        \\
                      &\le (BCd)^{|\be|+1}(|\ga|+1)^{|\be|} D^{(|\ga|+1)|\be|} |\be|!\, w^{[\xi]}_{(|\ga|+1)|\be|}
                      \\
                      &\le  e^{G/\xi} (BCd)^{|\be|+1}(|\ga|+1)^{|\be|} D^{(|\ga|+1)|\be|}  W^{[H\xi]}_{|\be|}
                      \\
                      &= A \si^{|\be|} W^{[H\xi]}_{|\be|},
    \end{align*}
    where $A= BCde^{G/\xi}$ and $\si = BCd(|\ga|+1) D^{|\ga|+1}$ and 
    $G,H > 1$ are constants depending only on $\om$ and $\ga$.

    A careful re-reading of the proof confirms the supplementary assertion on the boundedness.
    (Note that it is no restriction to assume $\xi\ge 1$ so that $e^{G/\xi} \le e^G$.)
\end{proof}

\begin{lemma} \label[l]{l:compact}
    Let $K\subseteq \R^d$ be a compact fat subanalytic set. 
    Let $\om$ be a weight function with associated family $\{W^{[\xi]}\}_{\xi>0}$.
    For each $\rh>0$ there exists $\si>\rh$ such that the inclusion $\cE^\om_\rh(K) \to  \cE^\om_\si(K)$ is compact.
    Recall that $\cE^\om_\rh(K) = \{ f \in \cC^\infty(K) : \|f\|^\om_{K,\rh} <\infty\}$, 
    see \Cref{d:Eom}.
\end{lemma}

\begin{proof}
    We adapt the arguments of \cite[Proposition 2.2]{Komatsu73}.
    We may assume that $K$ is connected.
    Then there exist an integer $m\ge 1$ and a constant $C>0$ such that 
    any two points $x,y \in K$ can be joined by a rectifiable path $\ga$ in $K$ with 
    length $\ell(\ga)$ satisfying $\ell(\ga) \le C |x-y|^{1/m}$ (see e.g.\ \cite[Theorem 6.10]{BM88}).

    Let $\cB$ be the unit ball in $\cE^\om_\rh(K)$ and $\ep>0$. 
    By \Cref{l:prop}(4), there exist $H,C> 1$ such that $2^k W^{[\rh]}_k \le C W^{[H\rh]}_k$ for all $k \in \N$. 
    Setting $\si := H\rh$, we have 
    \begin{equation} \label{e:Ac1}
        2^k W^{[\rh]}_k \le C W^{[\si]}_k \quad \text{ for all } k \in \N.
    \end{equation}
    Choose $n \in \N$ such that 
    \begin{equation} \label{e:A5}
        \frac{1}{2^n} \le \frac{\ep}{2C}.
    \end{equation}
    Let $f \in \cB$, $x,y \in K$, and $\ga$ a rectifiable path joining $x$ and $y$ with the listed properties.
    For $\al \in \N^d$ with $|\al|\le n$, we have
    \begin{align*}
        |\p^\al f(x) - \p^\al f(y)| &\le \sqrt d\, \ell(\ga) \sup_{z \in \ga, 1\le i\le d} |\p^{\al+e_i}f(z)| 
        \\
                                    &\le C \sqrt d\, |x-y|^{1/m} \|f\|^\om_{K,\rh} \exp(\tfrac{1}{\rh} \vh^*(\rh (|\al|+1)))
                                    \\
                                    &\le C_1 \, |x-y|^{1/m}.
    \end{align*}
    Thus $\{\p^\al f : f \in \cB\}$ is equicontinuous and pointwise bounded.
    By the theorem of Arzel\`a--Ascoli, 
    $\{\p^\al f : f \in \cB\}$ is relatively compact in $\cC(K)$. 
    So there exist $f_1,\ldots,f_k \in \cB$ such that for each $f \in \cB$ there is $i \in \{1,\ldots,k\}$
    such that
    \[
        \sup_{x \in K}|\p^\al f(x) -\p^\al f_{i}(x)| \le \ep \cdot \exp(\tfrac{1}{\si} \vh^*(\si |\al|))
    \]
    for all $|\al|\le n$. 
    For $|\al|>n$, we have  
    \begin{align*}
        \sup_{x \in K}|\p^\al f(x) -\p^\al f_{i}(x)| &\le (\|f\|^\om_{K,\rh} + \|f_i\|^\om_{K,\rh}) \exp(\tfrac{1}{\rh} \vh^*(\rh |\al|))
        \\
                                                     &\stackrel{\eqref{e:A5}}{\le} 2 \cdot \frac{2^n \ep}{2C} \cdot W^{[\rh]}_{|\al|} 
                                                     \\
                                                     &\stackrel{\eqref{e:Ac1}}{\le} \ep \cdot W^{[\si]}_{|\al|} = \ep \cdot \exp(\tfrac{1}{\si} \vh^*(\si |\al|)).
    \end{align*} 
    Thus $\{f_1,\ldots,f_k\}$ is an $\ep$-net for $\cB$ in $\cE^\om_{\si}(K)$.
\end{proof}

\subsection*{Funding and/or Conflicts of interests/Competing interests.}

This research was funded 
    by the Austrian Science Fund (FWF) DOI 10.55776/PAT1381823.
    For open access purposes, the author has applied a CC BY public copyright license 
    to any author-accepted manuscript version arising from this submission.

    The author has no conflicts of interests or competing interests to declare that are relevant to the content of this article.


\def\cprime{$'$}
\providecommand{\bysame}{\leavevmode\hbox to3em{\hrulefill}\thinspace}
\providecommand{\MR}{\relax\ifhmode\unskip\space\fi MR }
\providecommand{\MRhref}[2]{%
    \href{http://www.ams.org/mathscinet-getitem?mr=#1}{#2}
}
\providecommand{\href}[2]{#2}

\end{document}